\documentclass[11pt]{amsart}

\usepackage{upgreek,amssymb,amsmath,geometry,latexsym,graphics,graphicx,tabularx,shapepar,enumerate,
hyperref,stmaryrd,glossaries,color,tikz-cd}

\DeclareSymbolFontAlphabet{\amsmathbb}{AMSb}

\numberwithin{equation}{section}

\newcommand{\Div}{\operatorname{Div}}
\newcommand{\ZZ}{\mathbb{Z}}
\newcommand{\FF}{\mathbb{F}}
\newcommand{\CC}{\mathbb{C}}
\newcommand{\QQ}{\mathbb{Q}}
\newcommand{\RR}{\mathbb{R}}
\newcommand{\GG}{\mathbb{G}}
\newcommand{\NN}{\mathbb{N}}
\newcommand{\PP}{\mathbb{P}}
\newcommand{\TT}{\mathbb{T}}
\newcommand{\sgn}{\operatorname{sgn}}

\newtheorem{Theorem}{Theorem}[section]
\newtheorem{Lemma}[Theorem]{Lemma}
\newtheorem{Question}[Theorem]{Question}
\newtheorem{Proposition}[Theorem]{Proposition}
\newtheorem{Conjecture}[Theorem]{Conjecture}
\newtheorem{Corollary}[Theorem]{Corollary}
\newtheorem{Definition}[Theorem]{Definition}
\newtheorem{Remark}[Theorem]{Remark}

\title{Zeta functions over curves}

\date{\today}
\author{F. Pellarin, with an appendix by G.\,H. Ferraro}

\address{Giacomo H. Ferraro\\
Department of Mathematics\\
3368 TAMU\\
Texas A\&M University\\
College Station, TX 77843-3368, USA.}

\address{Federico Pellarin\\
Dipartimento di Matematica Guido Castelnuovo\\
Universit\`a Sapienza \\
Piazzale Aldo Moro 5 \\
00185 Rome, Italy.}


\begin{document}

\begin{abstract}
In this paper we review the theory that David Goss developed, starting from 1979, to construct zeta functions 
around Carlitz zeta values and other remarkable formal series in local fields of positive characteristic. In the description of Goss' theory, we will see how it is primarily motivated by analogies with the classical theory of complex valued zeta and $L$-functions. We compare Goss' theory with another way of constructing zeta and $L$-functions that emerged in more recent times. 
The functions in the second type have as domains curves over finite fields, with the scalars extended to complete and algebraically closed fields of positive characteristic. The second type of functions interacts with 
Goss' functions but remains fundamentally different. We shall review a rationality theorem of Ferraro that allows, among others, to introduce some kind of analogue of the function $\xi$ of Riemann. In the path of describing Ferraro's proof,
we present some essential tools useful to get into the theory: shtuka divisors and functions, special functions, Anderson motives, Drinfeld modules, among others. We discuss certain relative zeta functions that can be considered as counterparts of Dedekind zeta functions. In particular, we use methods introduced by Goss to prove that these functions extend to entire functions. 

The paper contains an appendix by Ferraro where the property of entireness of the above relative zeta functions is deduced (in a special case) from 
the conjunction of a formula by Angl\`es, Ngo Dac and Tavares Ribeiro and Ferraro's rationality formula. Ferraro also 
presents a conjecture on the order of vanishing of this function at the canonical point $\Xi$ and some numerical evidences.
\end{abstract}

\maketitle

\tableofcontents

\section{Introduction}

The aim of the present paper is to give an overview of certain results on ``zeta and $L$-values in global function fields.'' More precisely, we present results in the perspective of David Goss' attempt of developing a theory of zeta and $L$-functions interpolating the above values, which are formal series in local fields of positive characteristic. There are ``zeta and $L$-values'' in 
local fields of positive characteristic: the first and most famous example was first considered by Carlitz in 1935 and is described in \S \ref{glimpse-Carlitz}. 

The next logical step is ``zeta functions''. After having presented the interpolation problem in \S \ref{dream}, we review Goss' construction of zeta functions in \S \ref{Goss-theory-zeta-functions}. But there is no known analogue of functional equation for Goss zeta functions. Then we describe, in \S \ref{Zeta-functions-curves} and \S \ref{Drinfeld-shtuka-divisors} more recent results where another technique of interpolating zeta and $L$-values is presented, ``over curves'' with the scalars extended to 
complete algebraically closed fields (of positive characteristic). 

We will notably discuss a 2022 result of Ferraro (Theorem \ref{Ferraro-theorem}) providing a functional identity for these interpolating functions somehow reminiscent of the functional equation of Riemann's zeta function. A comparison with the classical theory of zeta functions, together with their zeta functions, will be given along the way. In \S \ref{Dedekind} we explain how Goss' theory allows to show that a certain analogue of Dedekind zeta function over a curve (relative zeta functions) extends to an entire function.
This result can be reached alternatively by using a different theory by Angl\`es, Ngo Dac and Tavares Ribeiro \cite{ANG&NGO&TAV}. This is studied, in a special case, in the appendix by Ferraro \S \ref{Appendix}. His use of the theory of \cite{ANG&NGO&TAV} allows him to reach more information on the structure of the divisor of zeroes of these functions. In particular, with the support of computational evidences, he formulates Conjecture \ref{multiplicity-conjecture} on the order of vanishing of these functions
at a certain canonical point $\Xi$ over the curve described in the text. We will not discuss this conjecture further in the present text leaving it for further developments we aim to investigate, but we mention that this agrees with the well known property that the order of vanishing at zero of the Dedekind zeta function associated to a number field $L$ equals the rank of the unit group
of the ring of integers of $L$. Therefore the potential importance of Ferraro's conjecture is that it opens the perspective of a 
factorization of these relative zeta functions over curves in terms of Artin-like $L$-functions, and class number formulas.

\subsection*{Acknowledgements} {\em To be added later}.

\section{A glimpse of Carlitz zeta values}\label{glimpse-Carlitz}

Carlitz zeta values are the first objects of the theory we want to describe.
In 1935 Carlitz \cite{CAR} proposes a formula claiming analogy with Euler's 1735 formula:
$$\sum_{n\geq1}\frac{1}{n^2}=\frac{\pi^2}{6}.$$
Carlitz' formula is
\begin{equation}\label{carlitz-formula}\sum_{\begin{smallmatrix}a\in\FF_q[\theta]\\ \text{$a$ monic}\end{smallmatrix}}\frac{1}{a^{q-1}}=\prod_{i\geq0}\left(1-\frac{\theta^{q^i}-\theta}{\theta^{q^{i+1}}-\theta}\right)^{q-1}.\end{equation}
Carlitz analyzes a more general class of formulas, just like Euler does, but let us stare at this one for a little while trying to interpret its meaning. 

On the left-hand side of \eqref{carlitz-formula}, the sum runs over the monic polynomials of the polynomial ring
$\FF_q[\theta]$ in the indeterminate $\theta$, coefficients in the finite field $\FF_q$ with $q$ elements and characteristic $p$.
To have such a formal series converging, we need to choose, over the fraction field $\FF_q(\theta)$ of $\FF_q[\theta]$, a multiplicative valuation $|\cdot|$ for which $\FF_q[\theta]$ is discrete. There is only one class of such a valuation corresponding to the infinity place of $\FF_q(\theta)$. The corresponding metric is determined by the condition $$|\theta|>1.$$
For this valuation, $\FF_q[\theta]$ is also co-compact, just like $\ZZ$ in $\RR$. 

The completion of $\FF_q(\theta)$ at this infinity place 
can be identified with the local field $\FF_q((\frac{1}{\theta}))$. The right-hand side converges for this valuation as well, but 
``much faster'' than the left-hand side. Experimentally, if we try to sum in the left it is convenient to collect together contributions coming from polynomials of same degree and then sum over degrees, then there are terms that formally cancel each other. So we already notice a benefit of Carlitz's formula: instead of computing truncations of the left-hand side
(more difficult if $q$ is a large power of the characteristic $p$), we can just approximate the fast converging
product on the right-hand side, where the factors are rational functions over $\FF_p$. Additionally, we can even lift the right-hand side to a formal series of $1+\frac{1}{\theta}\ZZ[[\frac{1}{\theta}]]$ and then reduce modulo $p$ the coefficients. It does not seem to be possible to do this directly with the left-hand side, without transforming it.

To have a first numerical taste we give an example: if $q=9$, a computer computation yields the truncation to the power $q^3$ of the uniformiser $\frac{1}{\theta}$ of $\FF_p((\frac{1}{\theta}))$, which is
\begin{multline*}
1+\frac{1}{\theta^{72}}+\frac{2}{\theta^{80}}+\frac{1}{\theta^{144}
   }+\frac{2}{\theta^{152}}+\frac{1}{\theta^{216}}+\frac{
   2}{\theta^{224}}+\frac{1}{\theta^{288}}+\frac{2}{\theta^{296}}+\frac{1}{\theta^{360}}+\frac{2}{\theta^{368}}+\frac{1}{\theta^{432}}+\frac{2}{\theta^{440}}+\frac{1}{\theta^{504}}+\\+\frac{2}{\theta^{512}}+\frac{1}{\theta^{576}}+\frac{2}{\theta^{584}}
   +\frac{1}{\theta^{648}}+\frac{2}{\theta^{656}}+\frac{2}{\theta^{720}}+\frac{1}{\theta^{728}},\end{multline*}
   where $0,1,2$ can be identified with the corresponding elements of $\FF_3\subset\FF_9$.   

There is no need to engrave these digits on a stone, their significance is somehow diminished by the fact that the choice of the uniformiser is anyway non-canonical (we choose here the uniformizer $\frac{1}{\theta}$ just because we used $\theta$ to define $\FF_q[\theta]$, but there are uncountably many choices). Depending on the choice, the above truncated expansions may artificially look simpler or conversely, more complicated. Yet it is possible to prove that
the right-hand side of (\ref{carlitz-formula}) determines an element of $\FF_q((\frac{1}{\theta}))$ which is transcendental over 
   $\FF_q(\theta)\subset\FF_q((\frac{1}{\theta}))$ (the first proof was obtained by Wade in 1941, see \cite{WAD}).

Carlitz's principal insight is that the series on the left-hand side of (\ref{carlitz-formula}) is a function field analogue of a value of Riemann's zeta function at an even positive integer, with the product on the right-hand side a rational (i. e. in $\FF_q(\theta)$) multiple of an even power of an ``analogue of the complex number $2i\pi$''. This is a first instance of a ``functional equation'' of a zeta value at a positive integer divisible by $q-1$.

\subsubsection*{More on Carlitz's zeta values}
The left-hand side of \eqref{carlitz-formula} also has a eulerian factorization. For $k\geq 1$, the following product expansion holds:
\begin{equation}\label{carlitz-zeta-values}
\zeta_{\FF_q[\theta]}(k):=\sum_{\begin{smallmatrix}a\in\FF_q[\theta]\\ \text{$a$ monic}\end{smallmatrix}}\frac{1}{a^k}=
\prod_{\begin{smallmatrix}P\in\FF_q[\theta]\\ \text{$P$ irreducible monic}\end{smallmatrix}}\left(1-\frac{1}{P^k}\right)^{-1}.
\end{equation}
There exists, for all $k\geq 1$, a functor called {\em $k$-th tensor power of the Carlitz module}
$$C^{\otimes k}:\{\FF_q[\theta]\text{-algebras}\}\rightarrow\{\FF_q[\theta]\text{-modules}\}$$ such that for all $P$ irreducible and monic in $\FF_q[\theta]$, the finite $\FF_q[\theta]$-module
$C^{\otimes k}(\FF_q[\theta]/P\FF_q[\theta])$ is isomorphic to the finite cyclic $\FF_q[\theta]$-module $\FF_q[\theta]/(P^k-1)$ in a way that is similar, in the case $k=1$, to the fact that $\GG_m(\FF_p)\cong\ZZ/(p-1)\ZZ$ as a group. For $k=1$ the above functor is called {\em Carlitz's functor} $C=C^{\otimes1}$, see the books \cite{GOS,PAPI} and $C^{\otimes k}$ is a higher dimensional variant introduced by Anderson and Thakur in \cite{AND&THA}, analogues of Tate's twists $\ZZ(k)$. Most of the properties of Carlitz's module we recall here are contained in \cite[\S 4.4]{PEL5}.

\subsubsection*{The field $\CC_\infty$ and the analogue of $2\pi i$}
We now consider a complete, algebraically closed field containing the zeta values $\zeta_{\FF_q[\theta]}(n)$.
Let $\CC_\infty$ be the completion of a separable closure of $\FF_q((\frac{1}{\theta}))$. So we can write:
$$\CC_\infty:=\widehat{\FF_q\Big(\!\Big(\frac{1}{\theta}\Big)\!\Big)^{\text{sep}}}$$
(the exponent denotes a separable closure, and the hat is the completion with respect to the unique extension of
the valuation of $\FF_q(\!(\frac{1}{\theta})\!)$). This field is algebraically closed and complete. Unlike $\CC$, it is not locally compact.
It is an
$\FF_q[\theta]$-algebra and therefore we can consider the $\FF_q[\theta]$-module $C(\CC_\infty)$. 

Let us inspect $C(\CC_\infty)$ more closely. The multiplication by $\theta$, determining the whole $\FF_q[\theta]$-module structure of $C(\CC_\infty)$, is explicitly determined by the 
$\FF_q$-endomorphism $\theta+\tau$, where $\tau$ is the $\FF_q$-automorphism $c\mapsto c^q$ over $\CC_\infty$. There is a more general but similar structure allowing to work with $C^{\otimes k}$, see \cite{AND&THA}. $C(\CC_\infty)$ comes equipped with the structure of a non-compact connected rigid analytic curve over $\CC_\infty$ which is analytically isomorphic, via an analogue of the classical exponential function called the {\em Carlitz exponential}, to the quotient
$\CC_\infty/\widetilde{\pi}\FF_q[\theta]$, which is an $\FF_q[\theta]$-module. These properties are proved in \cite[\S 4.4.2]{PEL5}. 

Everything can be computed explicitly! We have the formula: 
\begin{equation}\label{Carlitz-pi}
\widetilde{\pi}=\theta(-\theta)^{\frac{1}{q-1}}\prod_{i>0}\Big(1-\frac{\theta}{\theta^{q^i}}\Big)^{-1}\in\CC_\infty^{\times}
\end{equation}
(it is defined up to multiplication by an element of $\FF_q^\times$, therefore uniquely determined only in the case $q=2$). This is the above mentioned Carlitz's analogue of $2\pi i$, defined up to multiplication by an element of $\FF_q^\times$. Again in the case $q=9$ with the uniformizer $\frac{1}{\theta}$ the computer says that the truncation (this time to the $q^2$-th term) is
$$(-\theta)^{\frac{1}{q-1}}\Big(\theta+\frac{1}{\theta^7}+\frac{1}{\theta^{15}}+
   \frac{1}{\theta^{23}}+\frac{1}{\theta^{31}}+\frac{1}{\theta^{39}}+\frac{1}{\theta^{47}}+\frac{1}{\theta^{55}}+\frac{1}{\theta^{63}}+\frac{1}{\theta^{71}}+\frac{2}{\theta^{79}}\Big).$$
Finally, $\widetilde{\pi}$ and the left-hand side of (\ref{carlitz-formula}) are related:
$$\prod_{i\geq0}\Big(1-\frac{\theta^{q^i}-\theta}{\theta^{q^{i+1}}-\theta}\Big)^{q-1}=\frac{\widetilde{\pi}^{q-1}}{\theta-\theta^q},$$
this is easy to check and in fact a characteristic zero lift of this formula holds, but it can be also deduced from \cite[Lemma 3.2.1]{GOS}, see also \cite[Theorem 4.4.6 and Remark 4.4.7]{PEL5}. In this game of analogies initiated by Carlitz, the polynomial $\theta-\theta^q$  corresponds to the denominator $6$ in Euler's formula for $\zeta(2)$ and we can further push our search of analogue structures. Carlitz introduced function field analogues of the factorial and Bernoulli numbers  (they are elements of $\FF_q[\theta]$ and $\FF_q(\theta)$) in order
to obtain analogues of Euler's famous 1735 formulas
\begin{equation}\label{euler-formulas}
\zeta(2k)=\frac{(-1)^{k-1}B_{2k}}{2(2k)!}(2\pi)^{2k},\quad k\in\NN^*,\quad \frac{x}{e^x-1}=:\sum_{n=0}^\infty B_n\frac{x^n}{n!}.
\end{equation}
for $\zeta_{\FF_q[\theta]}(n)$, $(q-1)\mid n$, $n\geq1$ (more details in \cite{GOS}).

Carlitz further obtained an analogue of 
Clausen-von Staudt theorem. The snow ball of analogies continued to roll until recent times, notably with analogues of Herbrand-Ribet theorem \cite{TAE2,ANG&PEL&TAV}, or the analogue for ``Bernoulli-Carlitz'' elements of a conjecture in \cite{KAN} for the reduction modulo $p$ of Bernoulli numbers $B_{p-k}$ for $p$ a moving prime and $k$ a fixed odd integer $\geq 3$ (see also \cite{KAN&ZAG}), proved in \cite{ANG&NGO&TAV2}, and other results. 

\section{Goss' Dream}\label{dream}

Euler's formula for $\zeta(2)$ can also be deduced from the functional equation of Riemann's zeta function, which involves Euler's gamma function. Namely, the function of the complex variable $s$
$$\xi(s):=\pi^{-\frac{s}{2}}s(1-s)\Gamma\Big(\frac{s}{2}\Big)\zeta(s),$$
extends to an entire function over $\CC$, non-vanishing over $\RR$, and satisfies the {\em symmetric functional equation}
\begin{equation}\label{functional-equation-xi}
\xi(s)=\xi(1-s).\end{equation}
 
 We ask if  there is an analogue of Riemann's zeta function interpolating the Carlitz zeta values (\ref{carlitz-zeta-values}) that satisfies ``some kind of functional equation'' and if the formula
(\ref{carlitz-formula}) can be deduced from it. Unfortunately, while there are several candidates for zeta functions (see the specimens described in the paper), at the time of writing this, the author is unaware of such a development, this question remains at the moment unanswered; we touch at a point where analogies are superseded by deeper phenomena that are not completely understood. 

Nevertheless, this was one of Goss' dreams:

\medskip

\begin{center} {\em Deduce Carlitz formula (\ref{carlitz-zeta-values}) from a ``functional equation''.} \end{center}

\medskip

This is one of the motivations for a theory of $\zeta$ and $L$-functions that he initiated in 1979 \cite{GOS0} and that we are going to review now. 

Goss' exploration is also rooted in Tate's 1950 Thesis \cite{TAT}. It is not our task to go deep in this direction but Tate's viewpoint has some advantage in better motivating Goss' construction of his zeta and $L$-functions. 

\subsubsection*{Complex quasi-characters}
Tate interprets $\zeta$ functions not just as functions over complex numbers but as functions over spaces of 
``quasi-characters''. 
\begin{Definition}{\em Let $k$ be a completion of a global field $K$. 
A {\em quasi-character} for $k$ (in the sense of Tate) is a continuous group homomorphism 
$$k^\times\rightarrow\CC^\times.$$
}
\end{Definition}
 The target space of a quasi-character is chosen in this way because it contains classical zeta and $L$-values; it contains the real number $\zeta(2)$ for example.

Note that
\begin{equation}\label{structure-of-Ctimes}
\CC^\times=\RR_{>0}u
\end{equation}
where $u$ is the kernel of $|\cdot|$ (unit circle) and this help in computing groups of quasi-characters. If $\alpha\in\CC^\times$ then we can decompose, uniquely, $$\alpha=\tilde{\alpha}|\alpha|$$ with $\tilde{\alpha}\in u$ and $|\alpha|\in\RR_{>0}=e^\RR$. For example if
$k=\CC$ and $c$ is a quasi-character of $k^\times$,
$$c(\alpha)=\tilde{c}(\tilde{\alpha})|\alpha|^s$$ with $s\in\CC$ where $\tilde{c}:u\rightarrow u$ is a character. This allows, for example, to 
show that the group of quasi-characters $\CC^\times\rightarrow\CC^\times$ is isomorphic to 
$\ZZ\times\CC,$ where the elements in $\{0\}\times\CC$
are determined by the ``unramified'' (\footnote{This is the terminology introduced by Tate.}) quasi-characters $|\cdot|^s$ with $s\in\CC$; here we see one convincing reason for which $\zeta$ is a function of a complex variable.
More generally, we obtain an explicit description of quasi-characters of $k^\times$ for any completion $k$, see
\cite[\S 2.3]{TAT}. Given a number field $K$, Tate then studies quasi-characters over more general locally compact groups, such as the id\`ele group of $K$.

We can continue describing Tate's approach but Carlitz zeta values do not belong to $\CC$! 
We are rather interested in the topological group of $\CC_\infty$-valued quasi-characters
$$\operatorname{Hom}_c\Big(\FF_q\big(\!\big(\frac{1}{\theta}\big)\!\big)^\times,\CC_\infty^\times\Big).$$ 

\subsection{Setting and notation}
It is the moment to establish the perimeter of our discussion and to choose notations that we will adopt all along this paper.
Let $K$ be the function field of a geometrically connected smooth projective curve $X$ over $\FF_q$. We choose $\infty$ a place of $K$ of degree $d_\infty=1$ (hence we assume that $X$ is chosen so that it has at least one $\FF_q$-rational point, an hypothesis that can be removed, but it offers the advantage of an easier introduction in the theory).  We write 
\begin{equation}\label{A}
A=H^0(X\setminus\{\infty\},\mathcal{O}_X).
\end{equation} 
This is the subring of $K$ whose elements are regular away from $\infty$.  The completion $K_\infty$ of $K$ at $\infty$ is equipped with a multiplicative $\infty$-adic valuation $|\cdot|$. For $a\in A\setminus\{0\}$ we have 
$$|a|=c^{\deg(a)}$$
for a real number $c>1$ the choice of which determines $|\cdot|$ uniquely. 
The completion $\mathbb C_\infty$ of a fixed separable closure $K_\infty^{\text{sep}}$ of $K_\infty$ 
has a unique extension of the map $|\cdot|:K_\infty\rightarrow\RR_{\geq 0}$ to a multiplicative valuation. 
In $\CC_\infty$, $A$ is discrete; note also that $A$ is co-compact in $K_\infty$ (see \cite[Lemma 4.2.5]{PEL5}).
We fix a uniformizer $\pi$ of $K_\infty$ so that we can identify $K_\infty$ with the field of formal Laurent series 
$\FF_q((\pi))$. 

To introduce analogues of Carlitz zeta values \eqref{carlitz-zeta-values} we still need to define a notion of monic element in $A$. 

\begin{Definition}\label{sign-function}{\em Let $L/K_\infty$ be an algebraic extension, with residual field $\FF_L\subset\FF_q^{\text{alg}}$, identified with a subgroup of $L$.
A {\em sign function} for $L$ is a group homomorphism 
$$\sgn_L:L^\times \rightarrow \FF_L^\times$$  inducing the identity map over $\FF_L^\times\subset L^\times$.
}
\end{Definition}
Note that we are not asking $L$ to be a finite extension of $L_\infty$. Also, given any sign function $\sgn_L$, there is a unique 
extension to a group homomorphism $\sgn_L:\widehat{L}^\times\to\FF_L^\times$, where $\widehat{L}$ is the completion of $L$.
If $L=K_\infty$ we simplify the notation writing $\sgn$
instead of $\sgn_{K_\infty}$. In this case we define $A^+$ to be the set of {\em monic elements} of $A$, i.e., the set of $a \in A$ such that $\sgn(a)=1$. 

Here is a partial generalization of Carlitz's zeta values:
\begin{equation}\label{carlitz-general-A}
Z_A(n):=\sum_{d\geq0}\sum_{\begin{smallmatrix}a\in A^+\\ \deg(a)=d\end{smallmatrix}}\frac{1}{a^n}\in K_\infty,\quad n>0.
\end{equation}
The condition $n>0$ ensures convergence. The definition depends on the choice of $\sgn$. But let us consider the more general series
$$\widetilde{Z}_A(x,n):=\sum_{d\geq0}x^{-dn}\sum_{\begin{smallmatrix}a\in A^+\\ \deg(a)=d\end{smallmatrix}}\frac{1}{a^n}\in K[[x]],\quad n>0,$$ introduced in \cite[\S 2]{THA}. It can be evaluated at $x\in\CC_\infty^\times$ with $|x|\geq 1$.
By \cite[Proposition 7.2.3]{GOS} this, although depending on the selected $\sgn$, accounts for all the choices of sign function on $K_\infty$.
Indeed if $\sgn'$ is another sign function and $\pi$ is a uniformizer of $K_\infty$ in the kernel of $\sgn$, then $\sgn'(\pi):=\lambda$ with $\lambda\in\FF_q^\times$. Then, with $Z'_A(n)$ the zeta value associated to $\sgn'$,
$$Z'_A(n)=\widetilde{Z}_A(\lambda^{-n},n).$$
Note also that $Z_A(n)$ is independent of $\sgn$ in the case $(q-1)\mid n$.

\subsubsection{Example: the Carlitz case}\label{in-genus-0}

In the case when $g=0$ we come back to the settings of the Carlitz zeta values introduced earlier.
We consider the genus $0$ case where $X=\mathbb P^1$ and $\infty$ is an $\FF_q$-rational point on it. Then we can write $A=\FF_q[\theta]$  polynomial ring in $\theta$, a rational function over $\PP^1$ with a simple pole at $\infty$, regular everywhere else on $\PP^1$. In this case we have $K=\FF_q(\theta)$ and $K_\infty=\FF_q((1/\theta))$ the completion of $K$ at the place associated with $\infty$. Then, Carlitz zeta values agree with \eqref{carlitz-general-A}, they are all non-zero.

\subsection{The multiplicative structure of $\CC_\infty$}

We give an analogue of \eqref{structure-of-Ctimes} for $\CC_\infty$. The main result is Proposition \ref{proposition-uvw} which is 
essentially a reformulation of \cite[Lemma 4.2.10]{PEL5} and \cite[\S 1]{ANG&NGO&TAV}. Observe that $\CC_\infty$ contains an algebraic closure of
$\FF_q$ that we denote by $\FF_q^{\text{alg}}$ and that we identify with its residual field. 

\begin{Definition}
{\em A {\em $\QQ$-line} of $\CC_\infty^\times$ is a subgroup $\Lambda$ of $\CC_\infty^\times$ such that the map $|\cdot|$ induces an isomorphism with $|\CC_\infty^\times|$.}
\end{Definition}

A $\QQ$-line is totally ordered for $|\cdot|$ and uniquely divisible, hence torsion free. It is a $\QQ$-vector space in $\CC_\infty^\times$. 

We can construct $\QQ$-lines in the following way (see \cite[\S 1]{ANG&NGO&TAV}).
Set $x_1:=x$ with $x\in\CC_\infty^\times$ such that $|x|\neq 1$ and define a {\em division sequence} $(x_i)_{i\geq 1}$ by choosing, for all $n>1$, an element $x_n\in\CC_\infty^\times$ with $x_n^n=x_{n-1}$ (this can be done as $\CC_\infty^\times$ is algebraically closed). For $\rho\in\QQ$ we can write $\rho=\frac{m}{n!}$
with $m\in\ZZ$ and $n\geq 0$. We set $$x^\rho:=x_n^m.$$ 
This is well defined and independent of the choice of $m,n$.
Denote by $x^\QQ$ the set of all these elements of $\CC_\infty^\times$.

\begin{Lemma}
Any set $x^\QQ$ defined as above is a $\QQ$-line. Any $\QQ$-line has the above structure.
\end{Lemma}

\begin{proof}
 Consider two elements of $x^\QQ$, $y=x^\rho$ and $y'=x^{\rho'}$, with $\rho=\frac{m}{n!}$ and $\rho'=\frac{m'}{n'!}$. Without loss of generality we can suppose that $n=n'$. Then $\rho=\frac{m}{n!}$ and $\rho'=\frac{m'}{n!}$ and we can write
 $y=x_n^m$, $y'=x_n^{m'}$. We get $yy'=x_n^{m+m'}=x^{\rho+\rho'}$. A similar property holds regarding the inversion operation. Now, writing
 $r=|x|\neq 1$ one sees that the map $x^\QQ\rightarrow r^\QQ=|\CC_\infty^\times|$ defined by sending $x^\rho$ to $r^\rho$
 is a group isomorphism. Now let $\Lambda$ be a $\QQ$-line. It contains a unique sequence 
 $(x_i)_{i\geq 1}$ after a choice of $x=x_1\in\Lambda$ and setting $x_i^i=x_{i-1}$.
 Then we have the sequence of subgroups
 $$x_1^\ZZ\subset x_2^\ZZ\subset\cdots\subset\Lambda$$
 with $x_i^\ZZ$ which is of index $i$ in $x_{i+1}^\ZZ$. Clearly $\Lambda=\cup_ix_i^\ZZ$.
  \end{proof}

In $\CC^\times$ too, there are $\QQ$-lines ($\RR$-lines in fact). For example $\RR_{>0}=e^\RR$. 
%Given two $\QQ$-lines of 
%$\CC_\infty^\times$, if they meet in more than one point, then they have a group of the form $y^\ZZ$ with $|y|\neq1$ in the intersection. 
Given a $\QQ$-line $\Lambda$ and any element $x\in\Lambda\setminus\{1\}$, we write $\Lambda=x^\QQ$ (the notation is loose,  the group structure is uniquely determined after a suitable choice of a division sequence $(x_i)_i$).

Suppose we are given a sign function $\sgn_{\CC_\infty}$ for $\CC_\infty$. We can construct $x\in\CC_\infty^\times$ such that $|x|\neq1$ and $\sgn_{\CC_\infty}(x)=1$ in the following way. The group $|\CC_\infty^\times|\subset\RR_{>0}$ is infinite, so we can find $y\in \CC_\infty^\times$ with $|y|\neq1$ and choose
$$x=\frac{y}{\sgn_{\CC_\infty}(y)}$$
which lies in the kernel of $\sgn_{\CC_\infty}$.

We define the subgroups of $\CC_\infty^\times$ depending on the choice of the couple $(\sgn_{\CC_\infty},x)$:
\begin{eqnarray*}
u&=&\operatorname{Ker}(|\cdot|)\\
v&=&\operatorname{Ker}(\sgn_{\CC_\infty})\\
w&=&\{y\in\CC_\infty^\times:\exists n\text{ such that }y^n\in x^\ZZ\}
\end{eqnarray*}
(in fact $u$ does not depend on this choice).
Then $u$ is the {\em unit circle} of $\CC_\infty^\times$. If $\mathfrak{m}_{\CC_\infty}=\{z\in\CC_\infty:|z|<1\}$ is the maximal ideal corresponding to the valuation $|\cdot|$ in $\CC_\infty$, $u$ contains $1+\mathfrak{m}_{\CC_\infty}$, and equals the disjoint union 
\begin{equation}\label{union-gm}
\bigsqcup_{\zeta\in(\FF_q^{\text{alg}})^\times}\zeta+\mathfrak{m}_{\CC_\infty}
\end{equation}
 of translates of the open unit disc containing $0$. Both groups $1+\mathfrak{m}_{\CC_\infty}$ and $1+\mathfrak{m}_{K_\infty}$ are free $\ZZ_p$-modules and $1+\mathfrak{m}_{\CC_\infty}$ is even a $\QQ_p$-vector space as it is uniquely divisible. Additionally, these groups have infinite $\ZZ_p$-rank.
 
\begin{Proposition}\label{proposition-uvw}
With the above choice of $(\sgn_{\CC_\infty},x)$ we have:
\begin{enumerate}
\item $u\cap v=1+\mathfrak{m}_{\CC_\infty}$.
\item $u\cap w=(\FF^{\text{alg}}_q)^\times$.
\item $v\cap w$ is a $\QQ$-line $x^\QQ$ of $\CC_\infty^\times$.
\item We have
the direct product decomposition
\begin{equation}\label{decomposition-C-infty}
\CC_\infty^\times=(\FF_q^{\text{alg}})^{\times}\cdot x^\QQ\cdot (1+\mathfrak{m}_{\CC_\infty}).
\end{equation} 
\end{enumerate}
\end{Proposition}

\begin{proof}
(1) Since $u$ is the unit circle, it equals $(\FF^{\text{alg}}_q)^\times (1+\mathfrak{m}_{\CC_\infty})$. 
Note that $(1+\mathfrak{m}_{\CC_\infty})\cap (\FF^{\text{alg}}_q)^\times=\{1\}$. 
Every element of $1+\mathfrak{m}_{\CC_\infty}$ has sign one and the kernel of the sign function over 
$u$ equals $1+\mathfrak{m}_{\CC_\infty}$. (2) An element $y$ of
$w$ has valuation $|y|=1$ if and only if it is a root of unity because $|x|\neq1$ by hypothesis. (3) Consider $y,y'\in v\cap w$.
We have three possibilities: $|y|<|y'|$, $|y|>|y'|$ and $|y|=|y'|$. In the latter case $\zeta:=y/y'\in v\cap w$ satisfies $|\zeta|=1$ so it is a root of unity. But $$\sgn_{\CC_\infty}(\zeta)=\zeta=1$$
 because $z\in v$. This means that $y=y'$ and $v\cap w$ is totally ordered. Since it is uniquely divisible, it is a $\QQ$-line. (4) Observe that $u\cap v=1+\mathfrak{m}_{\CC_\infty},u\cap w=(\FF^{\text{alg}}_q)^\times$ and $v\cap w=:x^\QQ$ pairwise have intersection $\{1\}$. Consider $y\in\CC_\infty^\times$. We can write
$y=\sgn_{\CC_\infty}(y)\tilde{y}$ with $\tilde{y}\in v$. There exists a unique $\tilde{x}\in x^\QQ$ such that $|\tilde{x}|=|\tilde{y}|$.
Then $|\frac{\tilde{y}}{\tilde{x}}|=1$ and $\sgn_{\CC_\infty}(\frac{\tilde{y}}{\tilde{x}})=1$ (note that $\tilde{x}\in x^\QQ\subset v$) so that $\frac{\tilde{y}}{\tilde{x}}\in u\cap v=1+\mathfrak{m}_{\CC_\infty}$.
\end{proof}

If $\Lambda$ and $\Lambda'$ are $\QQ$-lines such that $\Lambda'\subset(1+\mathfrak{m}_{\CC_\infty})\Lambda$, then 
we have at once that $\Lambda\subset(1+\mathfrak{m}_{\CC_\infty})\Lambda'$ and this defines an equivalence relation $\sim$ over
$\QQ$-lines.
 
\begin{Corollary}
There is a bijective correspondence between sign functions for $\CC_\infty^\times$ and classes of equivalence of $\QQ$-lines of $\CC_\infty^\times$.
\end{Corollary}

\begin{proof}
Consider a $\QQ$-line $\Lambda$. It yields a direct product decomposition $\CC_\infty^\times=(\FF_q^{\text{alg}})^\times(1+\mathfrak{m}_{\CC_\infty})\Lambda$. Indeed, choosing $x'\in\CC_\infty^\times$ there exists exactly one element $x\in\Lambda$ such that 
$|x|=|x'|$. This means that $\frac{x'}{x}\in\operatorname{Ker}(|\cdot|)$. So we can write $x'\in(\FF_q^{\text{alg}})^\times(1+\mathfrak{m}_{\CC_\infty})x$. 

There is a unique group homomorphism $\CC_\infty^\times\rightarrow(\FF_q^{\text{alg}})^\times$ that reduces to the identity on $(\FF_q^{\text{alg}})^\times$ and has $(1+\mathfrak{m}_{\CC_\infty})\Lambda$ as a kernel. This is a sign function. We have constructed a map $F$ from the set of $\QQ$-lines to the set
of sign functions. This map is surjective. Indeed consider a sign function $\sgn$ for $\CC_\infty$. Choose $x\in\CC_\infty^\times$ with $|x|\neq1$ and $\sgn(x)=1$, set $x_1=x$. Inductively, we can construct a division sequence $(x_n)_{n\geq 1}\subset\operatorname{Ker}(\sgn)$
by setting $x_n=\frac{y}{\sgn(y)}$ where $y\in\CC_\infty^\times$ is such that $y^n=x_{n-1}$. In this way we construct a $\QQ$-line 
$\Lambda$ in the kernel of $\sgn$. Applying $F$ to this $\QQ$-line 
we get the original sign function $\sgn$. Finally, Suppose that $\Lambda$ and $\Lambda'$ are $\QQ$-lines
such that $F(\Lambda)=F(\Lambda')=\sgn$. Then $\Lambda,\Lambda'\subset\operatorname{Ker}(\sgn_{\CC_\infty})$ and $\Lambda\sim\Lambda'$.
\end{proof}

From Proposition \ref{proposition-uvw} we deduce
\begin{equation}\label{k-infty-decomposition}
K_\infty^\times\cong\FF_q^{\times}\cdot \pi^\ZZ\cdot (1+\mathfrak{m}_{K_\infty})
\end{equation}
with $\pi$ any uniformizer of $K_\infty$ and $\sgn$ is a sign function over $K_\infty^\times$ with $\sgn(\pi)=1$, so that $1+\mathfrak{m}_{K_\infty}=1+\pi\FF_q[[\pi]]$. If $\alpha\in K_\infty^\times$ we can write, uniquely, 
\begin{equation}\label{coordinates}
\alpha=\lambda\langle\alpha\rangle \pi^n
\end{equation}
where $\lambda,\langle\alpha\rangle, \pi^n$ are the projections of $\alpha$  respectively on $\FF_q^\times$, $1+\mathfrak{m}_{K_\infty}$ and $\pi^\ZZ$.

\subsubsection{$\CC_\infty$-valued quasi-characters}
We also deduce, in analogy with \cite[Theorem 2.3.1]{TAT}:
\begin{Proposition} Choose a sign function $\sgn$ for $K_\infty$ and a uniformizer $\pi$ such that $\sgn(\pi)=1$.
To any $\CC_\infty$-valued quasi-character $c\in \operatorname{Hom}_c(K_\infty^\times,\CC_\infty^\times)$ 
we can associate a unique triple $(\mu,\tilde{c},x)$ with $\mu$ endomorphism of $\FF_q^\times$, $\tilde{c}$ a continuous group homomorphism $1+\mathfrak{m}_{K_\infty}\rightarrow 1+\mathfrak{m}_{\CC_\infty}$ and $x\in\CC_\infty^\times$
such that if $\alpha$ is in $K_\infty^\times$ and decomposes as \eqref{coordinates}, then
$$c(\alpha)=\mu(\lambda)\tilde{c}(\langle\alpha\rangle)x^n.$$
\end{Proposition}

\begin{proof}
As $1+\mathfrak{m}_{K_\infty}$ and $\pi^\ZZ$ are torsion free, $c$ restricts to an endomorphism of $\FF_q^\times$ if evaluated on $\FF_q^\times$. 
It is enough to study the restriction of $c$ over the kernel of $\sgn$. Now the restriction $\tilde{c}$ of $c$ over $1+\mathfrak{m}_{K_\infty}$ has image in $1+\mathfrak{m}_{\CC_\infty}$. Finally, $\pi^\ZZ$ is cyclic infinite and 
its image is of the form $x^\ZZ$ with $x\in\CC_\infty^\times$.
\end{proof}

\begin{Remark}
{\em 
% we see that these are quite big groups. For instance, Galois theory yields that $\operatorname{End}((\FF_q^{\text{alg}})^\times)\cong\prod_{l\neq p}\ZZ_l$.
We quote a sentence of Tate in \cite[\S 2.3]{TAT} 

\begin{quote}
\ldots we shall find no model for the group of quasi-characters, or even for the group of characters, though such a model would be of the utmost importance.
\end{quote}

A similar problem arises for $\operatorname{Hom}_c(K_\infty^\times,\CC_\infty^\times)$.

\medskip

We point out that Jeong proved \cite{JEO} that 
the additive group of locally analytic endomorphisms of $1+\mathfrak{m}_{K_\infty}$ is isomorphic to the additive $\ZZ_p$ and 
the group of the locally analytic automorphisms of $1+\mathfrak{m}_{K_\infty}$ is isomorphic to $\ZZ_p^\times$.}
\end{Remark}

In \S \ref{Goss-theory-zeta-functions} we shall first give a synthesis of Goss' construction of zeta functions where he constructs certain topological subgroups of $\operatorname{Hom}_c(K_\infty^\times,\CC_\infty^\times)$. From \S \ref{Zeta-functions-curves} on we propose an alternative approach, where zeta functions are constructed as functions defined over $X$, with the scalars extended.

\section{Goss theory of zeta functions}\label{Goss-theory-zeta-functions}
We choose $(X,\infty,\sgn)$ so that we have the rings $A,K,K_\infty,\CC_\infty$ discussed previously. Choosing a uniformizer 
$\pi$ of $K_\infty$ with $\sgn(\pi)=1$ we have a decomposition \eqref{k-infty-decomposition}.

{\em Goss plane} is a topological subgroup $\mathbb{S}_\infty$ of $\operatorname{Hom}_c(K_\infty^\times,\CC_\infty^\times)$ containing a locally discrete copy of $\ZZ$ as a subgroup. This space is homeomorphic to $\CC_\infty^\times\times \ZZ_p$
and can be easily described by using the tools developed above and \eqref{coordinates}. 

Homeomorphisms are given associating 
$$s:=(x,y)\in\mathbb{S}_\infty\mapsto c_s:=\Big(\alpha=\lambda\langle\alpha\rangle\pi^n\mapsto \lambda\langle\alpha\rangle^y x^n\Big)\in \operatorname{Hom}_c(K_\infty^\times,\CC_\infty^\times)$$
and depend on the choice of a uniformizer $\pi$ of $K_\infty$. The group homomorphism $\ZZ\rightarrow\mathbb{S}_\infty$ is given by $n\mapsto (\pi^{-n},n)$  and also depends on that choice.
Goss used this group $\mathbb{S}_\infty$ as the space of ``complex exponents'' in his theory of zeta and $L$-functions, see his book \cite{GOS}. 

Fundamental in this is his notion of exponentiation by an exponent in $\mathbb{S}_\infty$. 

%In classical number theory there are several processes to define zeta and $L$-functions among which:
%\begin{enumerate}
%\item summing over subsets of discrete subgroups of $\CC$ (partial zeta functions)
%\item summing over ideals (Dedekind zeta function)
%\item eulerian products.
%\end{enumerate}

\subsection{Goss' exponentiation}\label{Goss-exponentiation-section}
We fix a uniformizer $\pi$ of $K_\infty$ and a sign function $\sgn:K_\infty^\times\rightarrow\FF_q^\times$ such that $\sgn(\pi)=1$.
Let $a$ be an element of $A^+$ (recall that $A^+=A\cap\operatorname{Ker}(\sgn)$). If $s=(x,y)\in\mathbb{S}_\infty$ Goss introduces the {\em exponentiation} of $a$ by $s$:
\begin{equation}\label{Goss-exponentiation}a^s:=c_s(a)=\langle a\rangle^y x^{\deg(a)}\in \CC_\infty^\times.
\end{equation}
If $s=(\pi^{-n},n)$ we get
$$a^s=\pi^{-n\deg(a)}\langle a\rangle^n=(\pi^{-\deg(a)}\langle a\rangle)^n=a^n,$$ because by \eqref{coordinates} $$a=\pi^{-\deg(a)}\langle a\rangle.$$

We can also define exponentiations of fractional ideals, hence providing group homomorphisms $\mathcal{I}_A\rightarrow\CC_\infty^\times$ where $\mathcal{I}_A$ is the group of invertible ideals of $A$. 
Consider a non-zero fractional ideal $I$ of $A$. Since the class group of $A$ is finite, there exists an integer $h>0$ such that
$$I^h=(a_I),$$ that is, $I^h$ is principal generated by $a_I\in A^+=A\cap \operatorname{Ker}(\sgn)$ (the choice is possible because $A^\times=\FF_q^\times$). 
Hence $h\deg(I)=\deg(a_I)$. We can use \eqref{coordinates} and write 
$$a_I=\langle a_I \rangle \pi^{-\deg(a)}.$$
The group $v$ in Proposition \ref{proposition-uvw} is uniquely divisible and we can set, for $s=(x,y)\in\mathbb{S}_\infty$,
\begin{equation}\label{I-to-the-s}
I^s:=\langle a_I \rangle^{\frac{y}{h}}x^{\deg(I)}\in\CC_\infty^\times
\end{equation}
(recall that $1+\mathfrak{m}_{\CC_\infty}$
is a $\QQ_p$-vector space for the multiplicative structure and we can identify $\frac{y}{d}$ with an element of $\QQ_p$).

\begin{Lemma}\label{lemma-ext-ideals}
For all $s\in\mathbb{S}_\infty$ the map $I\mapsto I^s$ defines a group homomorphism 
$\mathcal{I}_A\rightarrow\CC_\infty^\times$. Moreover, the 
field $V$ generated by $K$ and the values $I^{s_n}$ in $\CC_\infty$ with $s_n=(\pi^{-n},n)$ and $n\in\ZZ$ is a finite extension of $K$. 
\end{Lemma}

\begin{proof}
The first part of the lemma follows from Proposition \ref{proposition-uvw}. The second part is proved in \cite[Proposition 8.2.9]{GOS}.
\end{proof}
Other properties of this ideal exponentiation map are collected in \cite[\S 8.2]{GOS}. 

\subsubsection{Goss' partial zeta functions}\label{Goss-partial-zeta}
Choose a uniformizer $\pi$ and a sign function $\sgn$ for $K_\infty$ with $\sgn(\pi)=1$ as above. 
For a non-zero fractional ideal $I$ of $A$, we introduce the {\em partial zeta function}
\begin{equation}\label{definition-partial-zeta}
Z_I(s):=\sum_{d\geq 0}\sum_{\begin{smallmatrix}a\in I^+\\ \deg(a)=d\end{smallmatrix}}a^{-s}
\end{equation}
where $I^+$ denotes the set of elements of $I$ that are in the kernel of $\sgn:K_\infty^\times\rightarrow\FF_q^\times$.

The series $Z_I(s)$ converges for the valuation of $\CC_\infty$ if the ``real part of $s$ is $>1$''. The real part of $s$ is by definition the real number $|x|$ (observe that choosing any total order on $A^+$, we have that the sequence $a\mapsto \deg(a)$
tends to infinity). 
Once we choose a uniformizer $\pi$ of $K_\infty$, we have a group map $\ZZ\rightarrow\mathbb{S}_\infty$ and we can identify $\ZZ$ with its image. So we can write
\begin{equation}\label{goss-partial-zeta}
Z_I(n):=Z_I(s_n)=\sum_{\begin{smallmatrix}a\in I^+\end{smallmatrix}}a^{-s_n}\in K_\infty, \quad n\geq0
\end{equation}
where
\begin{equation}\label{definition-sn}
s_n:=(\pi^{-n},n)\in\mathbb{S}_\infty,
\end{equation}
and the values do not depend on $\pi$.  In the case $X=\PP^1$ with 
the usual point $\infty$ and the usual sign function, $A=\FF_q[\theta]$ and we are back to Carlitz zeta values \eqref{carlitz-general-A}. 

In fact, $Z_I$ defines a continuous function $\mathbb{S}_\infty\rightarrow\CC_\infty$, see below.

\subsection{Goss' zeta functions}
The {\em Goss'  zeta function} associated to the triple $(X,\infty,\sgn)$ is the formal series:
\begin{equation}\label{definition-Goss-zeta}
\zeta_A(s):=\sum_{\begin{smallmatrix}I\text{ non-zero }\\ \text{ ideal of }A\end{smallmatrix}}I^{-s}.\end{equation}
This series too, converges for $s=(x,y)\in\mathbb{S}_\infty$ with $|x|>1$. Indeed the sequence $I\mapsto \deg(I)$, running over the ideals of $A$, tends to infinity.
Additionally, over the same ``half-plane'' of Goss' plane, the zeta function $\zeta_A(s)$ 
has a converging eulerian product expansion
$$\zeta_A(s)=\prod_{P}\Big(1-\frac{1}{P^s}\Big)^{-1},$$
where the product runs over the maximal ideals $P$ of $A$, by Lemma \ref{lemma-ext-ideals}. Note that 
if $s=(x,y)$ is such that $|x|>1$, $P^s\neq1$ so that $\zeta_A(s)$ does not vanish for $|x|>1$. We have familiarity with such statements, for Riemann's zeta function, $\zeta(s)$ does not vanish for $s=x+iy$ with $x,y\in\RR$, $x>1$.

For many reasons this can be considered as an analogue of Riemann's zeta function attached to a triple $(X,\infty,\sgn)$. In this text we do not tackle more general $L$-functions, but we keep focussing on these objects to seize not only the analogies, but also their important differences if compared with the classical theory of zeta functions.

To investigate the first properties of $\zeta_A$ we decompose it in combination of partial zeta functions. We partition the group of invertible ideals $\mathcal{I}_A$ into distinct cosets $\mathcal{I}_1,\ldots,\mathcal{I}_n$ modulo 
the subgroup of the principal ideals, choosing, as representatives, invertible ideals $I_1,\ldots,I_n$ of $A$ such that $I_i^{-1}$ is an ideal of $A$ for all $i$. By Lemma \ref{lemma-ext-ideals} we have:
\begin{equation}\label{dec-total-partial}
\zeta_A(s):=\sum_{i=1}^nI_i^{-s}Z_{I_i^{-1}}(s),
\end{equation}
an identity that defines a continuous function over the ``half-plane'' of the elements $s=(x,y)\in\mathbb{S}_\infty$ such that $|x|>1$. For $A=\FF_q[\theta]$ as
above, using \eqref{dec-total-partial}, we find that 
$\zeta_A(s)$ agrees with $Z_A(s)$ and therefore, this function provides a continuous interpolation of Carlitz zeta values.
Properties of the partial zeta functions immediately imply corresponding properties for the zeta functions $\zeta_A$.

Goss also developed, more generally, a theory parallel to that of Dedekind zeta functions. Classically, Dedekind zeta functions are associated to
number field extensions $L/\QQ$. In Goss' theory, after the choice of a triple $(X,\infty,\sgn)$ one considers $E$ a finite extension of $K$ and $B$ the integral closure of $A$ in $E$. 
We have the norm map $N_{E/K}:\mathcal{I}_B\rightarrow\mathcal{I}_A$ which defines a group homomorphism from the 
group of the invertible ideals of $B$ to that of $A$. Then we define, for $s\in\mathbb{S}_\infty$:
$$\zeta_{E,A}(s):=\sum_{\begin{smallmatrix}\mathfrak{a}
\text{ ideal of }B\end{smallmatrix}}N_{E/K}(\mathfrak{a})^{-s}=\prod_{\begin{smallmatrix}\mathfrak{p}
\text{ maximal ideal of }B\end{smallmatrix}}\Big(1-\frac{1}{N_{E/K}(\mathfrak{p})^{s}}\Big)^{-1}$$
that are convergent series for $s=(x,y)$ with $|x|>1$. The key tool is again the exponentiation of ideals in $\mathcal{I}_A$ that we discussed in \S \ref{Goss-exponentiation-section}.

For the partial zeta functions $Z_I$ one of the most important properties, discovered by Goss, is:

\begin{Lemma}\label{analytic-extension-goss}
For all $I$ non-zero ideal of $A$, $Z_I$
extends to a continuous function $\mathbb{S}_\infty\rightarrow\CC_\infty$.
\end{Lemma} 
A detailed proof can be found in \cite{GOS}, see also \cite{THA0}. We omit the details also because in the present paper, inspired by Goss' work, we will give a proof of one of
these entireness results, for a different class of functions, see Theorem \ref{zeta-E-A-are-entire}, where we use tools of essentially the same type.

Lemma \ref{analytic-extension-goss} can be briefly explained by the fact that finite sums of the type
$$S_{I,d}(s):=\sum_{\begin{smallmatrix}a\in I^+\\ \deg(a)=d\end{smallmatrix}}a^{-s},$$ if non-zero,
may have large vanishing finite sub-sums as it follows from the proof of \cite[Theorem 8.9.2]{GOS}, and consequently, for fixed
$s$ with $d\rightarrow\infty$, they have multiplicative valuation going to zero faster than what trivial estimates could detect. This is something that occurs in our characteristic $p>0$ setting. Also the more general Dedekind-like
functions $\zeta_{E,A}$ introduced above define continuous functions $\mathbb{S}_\infty\rightarrow\CC_\infty$.

In \cite[Definition 8.5.1]{GOS} Goss introduces the notion of {\em entire function over $\mathbb{S}_\infty$}, a continuous $\ZZ_p$-family of $\CC_\infty$-valued entire functions of the variable $x^{-1}$ the members of which are series that are uniformly convergent on bounded subsets of $\CC_\infty$.
Goss interprets the existence of entire extensions of his zeta functions as one of the ``signs confirming the possible existence of functional equations''.

\subsubsection{Negative values}\label{negative-values}
There are elements $s\in\mathbb{S}_\infty$ such that $Z_I(s)$ vanish regardless of the choice of $I$. We have a phenomenon of ``trivial zeroes at negative values'' that is described in the following result.
\begin{Lemma}\label{lemma-partial-negative}
For all $n\leq 0$ we have $Z_I(n)\in A$. Moreover, $Z_I(0)=1$ if $I$ is trivial, while otherwise $Z_I(0)=0$, and if $q-1\mid n$ with $n<0$, then
$Z_I(n)=0$.
\end{Lemma} 

\begin{proof}
Write $I^+(d)$ for the set of elements of $I$ that are monic (that is, in the kernel of
$\sgn$, for the choice we did initially) and of degree $d$. Similarly, write $I(<d)$ for the $\FF_q$-vector space of the elements of $I$ that 
have degree $<d$. Then either $I^+(d)=\{0\}$ (this for example happens if $I=A$ and $d$ is a Weierstrass gap) or   
$I^+(d)=a_0+I(<d)$ for all $a_0\in I^+(d)$. By \cite[Lemma 8.8.1]{GOS} (see also \cite[Theorem 5.3.1]{THA0})
$$S_{I,d}(-n)=\sum_{a\in I^+(d)}a^n=\sum_{a\in I(<d)}(a_0+a)^n$$
vanishes for all $d$ such that $q-1$ times the dimension of $I(<d)$ is strictly larger than $n$. In particular the sequence of 
partial sums \eqref{definition-partial-zeta} stabilizes when $s=s_n$ for $n\leq 0$.
This proves that 
$Z_I(n)\in A$ for all $n\leq 0$. Now, suppose additionally that $q-1\mid n$ with $n<0$. Then, by the fact that $\FF_q^\times=A^\times$
(we use the hypothesis that $\infty$ has degree one),
$$S_{I,d}(n)=(-1)^n\sum_{a\in I(<d+1)}a^n.$$
By the previously proved property, 
$$Z_I(-n)=(-1)^n\sum_{a\in I(<d+1)}a^n$$
for all $d$ large enough (the sum on the right does not depend on $d$ if $d$ is large), but this sum vanishes if $n<(q-1)\dim_{\FF_q}(I(<d+1))$ again by Lemma \cite[Lemma 8.8.1]{GOS} and we can find $d$ such that this inequality holds, while the explicit special values $Z_A(0)=1$ etc. follow easily from the fact that the only set $I^+(d)$ that is non-empty and such that $q$ does not divide its cardinality is $A^+(0)=\{1\}$.
\end{proof}

By \eqref{dec-total-partial} and Lemma \ref{lemma-partial-negative} we deduce:

\begin{Corollary}\label{corollary-trivial zeroes}
The Goss zeta function $\zeta_A:\mathbb{S}_\infty\rightarrow\CC_\infty$ vanishes over the image of $-(q-1)\ZZ_{>0}=\{1-q,2-2q,\ldots\}$ in $\mathbb{S}_\infty$. 
\end{Corollary}

Recall that $Z_I,\zeta_A$ also depend on the choice of a uniformizer of $\pi$, but the above vanishing property does not.
Analogously, Riemann's zeta function, seen as a meromorphic function over $\CC$, has zeroes at even integers $n<0$, so called 
{\em trivial zeroes}. The fact that Riemann's $\xi(s)$ is entire without real zeroes follows from the interaction of 
the trivial zeroes with the poles of $\Gamma$. Goss raised the question of completing $\zeta_A$ multiplying it by some ``gamma factor'' 
to obtain 
an analogue of the function $\xi$ over $\mathbb{S}_\infty$. There are several functions that 
have something to share with Euler's gamma function, reviewed in Goss' book \cite{GOS}. However, the raised question remains at the moment unanswered. For instance, we do not know how to answer the following

\begin{Question}
Given $(X,\infty,\sgn)$, characterize the set $\{n\in\ZZ_{\leq0}:\zeta_A(n)=0\}$.
\end{Question}

Goss proves \cite{GOS0} that in the case $A=\FF_q[\theta]$,  
$$\{n\in\ZZ_{<0}:\zeta_A(n)=0\}=-(q-1)\ZZ_{>0}.$$ A more general result is discussed in Thakur's \cite[Theorem 5.3.2]{THA0}.
However, it does not answer the above question completely.

\subsubsection{Positive values}\label{positive-values} Given a triple $(X,\infty,\sgn)$, with $A$ the ring of functions over $X$ that are regular away from $\infty$, the values of Goss zeta function
\begin{equation}\label{definition-goss-zeta-values}
\zeta_A(n)=\sum_{\begin{smallmatrix} I \text{ ideal of }A\end{smallmatrix}}I^{-s_n}=\prod_{P}\Big(1-\frac{1}{P^{s_n}}\Big)^{-1}\in 1+\mathfrak{m}_{K_\infty}\subset K_\infty^\times,\quad n\geq 1
\end{equation}
(where $s_n$ is defined in \eqref{definition-sn} and the product runs over the maximal ideals $P$ of $A$) are recognized as the analogues of the values of Riemann's zeta function at the positive integers $\geq 2$.
Goss pioneered the study of these values with $n>0$, $q-1\mid n$ and found partial analogues of Euler's formulas \eqref{euler-formulas} (that he called {\em functional equations at even $n$}), see \cite[Theorem 8.19.4]{GOS}. There is quite a long story besides the next result:

\begin{Theorem}[Goss]\label{Goss-formula-Euler}
There exists an element $\widetilde{\pi}\in\CC_\infty^\times$, defined up to multiplication by an element of $\FF_q^\times$ and depending on the choice $(X,\infty,\sgn)$, such that
if $q-1\mid n$ and $n>0$, then
$$\zeta_A(n)=\alpha_n\widetilde{\pi}^n,$$
where $\alpha_n$ is an element of the compositum $VH$ in $\CC_\infty$ of $V$ the field described in Lemma \ref{lemma-ext-ideals} and 
$H$ the Hilbert class field of $K$.
\end{Theorem} 

Without giving full details, we mention that one of the key points in Goss' proof is that the choice $(X,\infty,\sgn)$ corresponds to a choice of a {\em sign-normalized Drinfeld module of rank one} $\phi$. This is, just as Carlitz's functor, a functor from $A$-algebras to $A$-modules. In the special case of the usual triple $(\PP^1,\infty,\sgn)$ in \S \ref{in-genus-0}, it reduces to Carlitz's module $C$. Then, $\phi(\CC_\infty)$ carries a structure of rigid analytic curve and it is analytically isomorphic to
$\CC_\infty/\widetilde{\pi} A$. In full generality, it is hard to determine explicit formulas for $\widetilde{\pi}$. But for certain choices of $(X,\infty,\sgn)$ it is nevertheless possible. This can notably be made in the case of $X$ an elliptic curve over $\FF_q$ determined by 
a Weierstrass model, with its point at infinity, see \cite{GRE&PAP} for this and other fascinating explicit formulas. 
The presence of the compositum $VH$ in the statement is due to the process of ideal exponentiation. Goss proves that for every invertible ideal $I$ of $A$, if $q-1\mid n$ and $n>0$, \begin{equation}\label{partial-result-of-goss}Z_I(n)\in H\widetilde{\pi}^n.
\end{equation}
Theorem \ref{Goss-formula-Euler} follows directly from this, jointly with \eqref{dec-total-partial}.

As a last comment on Goss theory, we mention that the structures described by Corollary \ref{corollary-trivial zeroes} and Theorem \ref{Goss-formula-Euler} seem to point toward ``something like functional equations''. However, at the moment we are far from obtaining formulas of the type (\ref{functional-equation-xi}). It seems that even finding hypothetical candidates for gamma factors, defined as inverses of continuous functions over $\mathbb{S}_\infty$, is an unsolved problem. 

%\begin{Remark}\label{remark-on-d-infty}{\em  HERE REMARK ON $d_\infty$}\end{Remark}

\section{Zeta functions over curves}\label{Zeta-functions-curves}

In this second part of our paper we discuss recent investigations on alternative ways of interpolating Goss' zeta values \eqref{definition-goss-zeta-values}. We will be mainly interested in interpolating special values of Goss' partial zeta functions
as in (\ref{goss-partial-zeta}) for example. One of the properties that we study here is that these values can be interpolated by rigid analytic functions over the curve $X$, with the scalars extended to $\CC_\infty$ (with essential singularities). The picture that emerges escapes, at least at first sight, from Goss' viewpoint and from Tate's game of quasi-characters but in fact turns out to be connected with it in a subtle way. It looks like the premises of a theory parallel to that of Dedekind zeta functions appear in this direction. Some examples of such functions are discussed in \S \ref{Dedekind} The experimental data observed by Ferraro in the appendix to the paper point toward the presence of many of the principal milestones of the classical theory.

We are going to present
results of Ferraro (see Theorem \ref{Ferraro-theorem}, and the more general results in \cite{FER1,FER2}) which can be interpreted as a {\em functional identity} for a variant of
Goss partial zeta functions discussed in \S \ref{Goss-partial-zeta}. In \S \ref{Drinfeld-shtuka-divisors} we present the tools used and we give an essentially self-contained sketch of Proof of Ferraro's theorem \ref{Ferraro-theorem} following his paper \cite{FER1}.

\subsection{Functions over curves}\label{partial-zeta-functions}

 We consider a triple $(X,\infty,\sgn)$ as in the previous sections. We continue to assume that $\infty$ is an $\FF_q$-rational point of $X$. 
 
 If $B$ is any ring with unit, there exists a unique ring homomorphism $\ZZ\rightarrow B$. 
 But if $B$ is a ring containing $\FF_q$, there can be several $\FF_q$-algebra homomorphisms $A\rightarrow B$ that can be encompassed in analytic families; we would like to make them into variables of ``zeta functions''.
For instance, there is a bijective correspondence
$$\{f:A\rightarrow\CC_\infty:\FF_q\text{-algebra homomorphisms }\}\leftrightarrow\{\text{Closed points of } \operatorname{Spec}(\CC_\infty\otimes_{\FF_q}A)\}.$$
Let us write
$$X_{\CC_\infty}:=X\times_{\operatorname{Spec}(\FF_q)}\operatorname{Spec}(\CC_\infty)$$
for the curve $X$ with the scalars extended to $\CC_\infty$.
The point $\infty$ extends to a point of $X_{\CC_\infty}$. We can identify
$$X_{\CC_\infty}\setminus\{\infty\}=\operatorname{Spec}(A\otimes_{\FF_q}\CC_\infty).$$ Points of this affine curve are therefore
in correspondence with $\FF_q$-algebra morphisms
$$\chi:A\rightarrow\CC_\infty.$$
The curve $X_{\CC_\infty}$ also has a canonical point $\Xi$, distinct from $\infty$, determined by the canonical embedding $A\subset\CC_\infty$
($\CC_\infty$ is constructed out of $A$ by successively constructing $K$ its fraction field, $K_\infty$ the completion at the place $\infty$, and then, the completion $\CC_\infty$ of a separable closure). 

There are other elementary tools that we want to use. The $q$-Frobenius
$F$ is the bijective map $X(\CC_\infty)\rightarrow X(\CC_\infty)$ which is the identity on $\{\infty\}$ and is otherwise determined by the map
$\chi\mapsto\chi^{(1)}$ where, with $\chi$ an $\FF_q$-algebra homomorphism, $\chi^{(1)}$ the $\FF_q$-algebra map determined by
$$\chi^{(1)}(a):=\chi(a)^q,\quad a\in A.$$ 
On $A\otimes_{\FF_q}\CC_\infty$ we also have the {\em Anderson twist}. This is the unique $\FF_q$-algebra automorphism
$A\otimes_{\FF_q}\CC_\infty$ determined by $a\otimes c\mapsto a\otimes c^q$. If $h\in A\otimes_{\FF_q}\CC_\infty$ we denote by $h^{(1)}$ its Anderson twist, it uniquely extends to rational functions over $X_{\CC_\infty}$. Similarly, we can define $h^{(n)}$ for all $n\in\ZZ$, because the Anderson twist is bijective. Note that $A\otimes_{\FF_q}\CC_\infty$ is the ring of rational functions over $X_{\CC_\infty}$ that are regular away from $\infty$. The evaluation of such a rational function
at a point $\chi\in X_{\CC_\infty}\setminus\{\infty\}$ can be done once clarified what it means for elementary tensors. If $a\in A$ and $c\in\CC_\infty$, we have 
$$(a\otimes c)(\chi)=c\chi(a)\in\CC_\infty.$$
Then the twist of Anderson is the unique $A\otimes1$-linear map $h\mapsto h^{(1)}$ over $\operatorname{Frac}(A\otimes_{\FF_q}\CC_\infty)$ such that
$$h^{(1)}(\chi^{(1)})=h(\chi)^q,\quad h\text{ regular function},\quad \chi\text{ point.}$$

\subsubsection{Tate algebras}

On the ring $A\otimes_{\FF_q}\CC_\infty$ we have the {\em Gauss valuation} extending $|1\otimes(\cdot)|$
over $1\otimes\CC_\infty$ (to distinguish it from $|\cdot|$ we shall denote it by $\|\cdot\|$). The completion for this valuation
$$\TT_A:=\widehat{A\otimes_{\FF_q}\CC_\infty}$$
is the {\em Tate algebra} associated to $A$. Its maximal spectrum can be identified with the set $\boldsymbol{D}_X$ of points $\chi$ of 
$X(\CC_\infty)\setminus\{\infty\}$ defined by:
\begin{equation}\label{DX}
\boldsymbol{D}_X=\{\chi:\{|(a\otimes1)(\chi)|:a\in A\}\text{ is bounded}\}.
\end{equation}
It can be proved that $\chi\in\boldsymbol{D}_X\setminus X(\FF_q^{\text{alg}})$ if and only if $(a\otimes1)(\chi)\rightarrow0$ as $a$ runs in $A$. Moreover
we can establish, by evaluation, a $\CC_\infty$-algebra isomorphism between $\TT_A$ and the algebra of the series $\sum_ic_ia_i$
with $(a_i)_i$ an arbitrarily chosen basis of the $\FF_q$-vector space $A\otimes1$ and $c_i\rightarrow0$, by noticing that any such series can be 
evaluated at elements of $\boldsymbol{D}_X$.
 This subset $\boldsymbol{D}_X$ of $X_{\CC_\infty}\setminus\{\infty\}$ can be identified with an affinoid of the rigid analytic curve $X^{\text{an}}_{\CC_\infty}$, the analytification of $X_{\CC_\infty}$; it does not contain $\infty$ and $\Xi$.
An element $f$ of $\TT_A$ belongs to $\TT_A^\times$ if and only if, seeing it as a function $\boldsymbol{D}_X\rightarrow\CC_\infty$, $f(\chi)\neq0$ for all $\chi\in \boldsymbol{D}_X$.
For example, in the Carlitz case of \S \ref{in-genus-0}, we can identify, in $A\otimes_{\FF_q}\CC_\infty$, $1\otimes A$ with $\FF_q[\theta]$ and $A\otimes 1$ with $\FF_q[t]$ where $t$ is the unique regular function over the affine line $\mathbb{A}_{\CC_\infty}^1$
such that for all points $\chi$, $t(\chi)=\chi$. In this case $(t^i)_{i\geq0}$ is a basis of $A\otimes1$.

\subsection{Partial zeta functions over curves}

Let $I$ be a non-zero ideal of $A$. We introduce, for $n\in\ZZ$, the formal series:
$$Z_{X,I}(n):=\sum_{d\geq 0}\sum_{\begin{smallmatrix}a\in I^+\\ \deg(a)=d\end{smallmatrix}}a\otimes a^{-n}.$$
If $n\geq 1$, we see that this defines an element of $\TT_A$ (the series clearly converges for the Gauss valuation of $\TT_A$). Therefore $Z_{X,I}(n)$ can be viewed as a function $\boldsymbol{D}_X\rightarrow\CC_\infty$ by setting
$$Z_{X,I}(n)(\chi):=\sum_{a\in I^+}\frac{a(\chi)}{a^n}\in\CC_\infty.$$ An analogue of Lemma \ref{analytic-extension-goss} holds for these functions, a proof in the particular case $n=1$ can be found in \cite[Lemma 1.1]{CHU&NGO&PEL}, but the general case can be tackled with the same methods, that fundamentally draw us back to \cite[Lemma 8.8.1]{GOS}, therefore we omit the details of the proof. Note that results of this kind already appeared in \cite{ANG&PEL}.

\begin{Lemma}\label{partial-zeta-are-entire}
For all $I$ ideal of $A$ and $n\geq 1$, the function $Z_{X,I}(n)$ extends to a rigid analytic function over $X_{\CC_\infty}\setminus\{\infty\}$ with an essential singularity at $\infty$.
\end{Lemma}

We will more simply say that $Z_{X,I}(n)$ extends to an {\em entire function}.
Variants of these functions where the variable is in $X_{\CC_\infty}^s$ with $s\geq 1$ can be considered, see \cite{ANG&PEL} in the Carlitz case and several papers authored by Angl\`es, Ngo Dac and Tavares Ribeiro among which we mention \cite{ANG&NGO&TAV}, containing many references. Multiple zeta variants are also considered in \cite{CHU&NGO&PEL}.

\subsubsection{Link with quasi-characters}

We would like to connect this construction with the discussion of \S \ref{Goss-theory-zeta-functions}. 
We see here that our functions are not simple re-parametrizations of Goss' functions, although they interact with them.
It is possible indeed to establish a correspondence between points $\chi\in\boldsymbol{D}_X\setminus X(\FF_q^{\text{alg}})$ and $\CC_\infty$-valued quasi-characters via the next Lemma.

\begin{Lemma}
Given $\chi\in\boldsymbol{D}_X\setminus X(\FF_q^{\text{alg}})$ we can find $c_\chi\in\operatorname{Hom}_c(K_\infty^\times,\CC_\infty^\times)$ extending the multiplicative map $A\rightarrow\CC_\infty$ defined by the evaluation at $\chi$.
\end{Lemma}

\begin{proof}
First note that the $\FF_q$-algebra map $A\rightarrow\CC_\infty$ defined by $\chi$ is injective because $\chi\not\in X(\FF_q^{\text{alg}})$ by hypothesis. Hence it extends to a field embedding $K\rightarrow\CC_\infty$.
For $n$ large the line bundle defined by $A(\leq n)\otimes_{\FF_q}\CC_\infty$ is very ample and determines 
a closed immersion $i$ of $X_{\CC_\infty}$ in $\PP^N$ for some $N$. Say that $i=(a_0:a_1:\cdots:a_N)$ with $\deg(a_0)=\deg(a_1)+1$ (this is possible with $n$ large) with $i(\infty)=(1:0:\cdots:0)$. We can make 
additional choices so that $0<|a_1(\chi)|<|a_0(\chi)|$. On one side $u:=\frac{a_1}{a_0}$ uniformizes $K$ at infinity
so that every element $y$ of $K_\infty$ can be expanded as a Laurent series $y=\sum_iy_iu^i\in\FF_q((u))$, on the other side
the series $y(\chi):=\sum_iy_ix^i$ with $x:=\frac{a_1(\chi)}{a_0(\chi)}\in\CC_\infty^\times$ such that $0<|x|<1$
converges in $\CC_\infty$ (to an element of its valuation ring) and this map is the quasi-character $c_\chi$ that we wanted.
\end{proof}

Conversely, for any {\em evaluative quasi-character} $c:K_\infty^\times\rightarrow\CC_\infty^\times$ defined by sending a chosen uniformizer $\pi$ to $x\in\CC_\infty$ with $0<|x|<1$, we can find a point $\chi$ and 
a uniformizer of $K_\infty$ in $K$ such that $c$ can be extended to $\operatorname{Hom}_c(K_\infty^\times,\CC_\infty^\times)$ and arises from the $\FF_q$-algebra map $\chi$ following the above construction. 

The entailed subspace of $\operatorname{Hom}_c(K_\infty^\times,\CC_\infty^\times)$ does not equal any of the spaces $\mathbb{S}_\infty$ of Goss. 
To explain the relationship with Goss' $Z_I$ we use a generalization of Goss' zeta functions introduced by Angl\`es
and discussed in \cite[Example 5.12]{GOS3} (it is only considered in the Carlitz case of \S \ref{in-genus-0} but it easily generalizes to our settings). Let us choose a uniformizer $\pi$ for $K_\infty$ and, for $n\geq 0$ let us introduce the set
$$\mathbb{S}_{n,\infty}^\sharp:=\CC_\infty^\times\times\mathbb{X}_\infty^n,$$
where $$\mathbb{X}_\infty=\mathfrak{m}_{\CC_\infty}\times\ZZ_p.$$ The case $n=1$ will be sufficient for our purposes.
If $s=(z,(x_1,y_1),\ldots,(x_n,y_n))\in\mathbb{S}_{n,\infty}^\sharp$ we set, with $a\in A^+$,
\begin{equation}\label{exponentiation-bis}
a^s:=z^{\deg(a)}\langle a\rangle_{x_1}^{y_1}\cdots\langle a\rangle_{x_n}^{y_n},
\end{equation}
where for $x\in\mathfrak{m}_{\CC_\infty}$, $\langle a\rangle_{x}$ is the result in $1+\mathfrak{m}_{\CC_\infty}$ of evaluating the formal series $\langle a\rangle$ replacing $\pi$ with $x$ (this is well defined, it is in the image of an $\FF_q$-algebra map).
The sets $\mathbb{S}_{n,\infty}^\sharp$ have no simple group structures but the images under the projections that forget all entries in $\mathfrak{m}_{\CC_\infty}$ are groups. If $s=(z,(x,y))$
we set $-s:=(z^{-1},(x,-y))$.
Note that we can identify $\mathbb{S}_{\infty}$ with the subset $\CC_\infty^\times\times\{\pi\}\times\ZZ_p$ of 
$\mathbb{S}_{1,\infty}^\sharp$. Comparing with 
\eqref{Goss-exponentiation} we see that $a^{(z,y)}=a^{(z,(\pi,y))}$ (the former is Goss' exponentiation while the latter is the one we defined in \eqref{exponentiation-bis} with exponents in
$\mathbb{S}_{1,\infty}^\sharp$). We see that, with $\chi$ such that $$\frac{a_1(\chi)}{a_0(\chi)}=x\in\mathfrak{m}_{\CC_\infty}$$
and $\frac{a_1}{a_0}\in K$ that can be identified with a uniformizer of $K_\infty$,
$$Z_{X,A}(\chi)=\sum_{d\geq 0}\sum_{\begin{smallmatrix}a\in A^+\end{smallmatrix}}a^{-s}$$
where $s=(\frac{x}{\pi},x,1)\in\mathbb{S}_{1,\infty}^\sharp$. Hence some special values of one of our functions, $Z_{X,A}$,
are also special values of these functions closely related to Goss' theory, that have been first considered by Angl\`es. Choosing $x_1=\cdots=x_n=\pi$ and ``moving'' the parameters $z,y_1,\ldots,y_n$ is an operation typical of the realm of Goss' functions and in \cite[Proposition 5.13]{GOS3} it is proved that there indeed is analytic continuation in the sense of Goss, while for us it is typical to leave $y_1,\ldots, y_n$ fixed, and ``move'' $z,x_1,\ldots,x_n$ (in fact $\chi$). Lemma \ref{partial-zeta-are-entire} 
stipulates that such functions too, allow analytic continuation, but in the rigid analytic sense (see also \cite[Lemma 5.17]{GOS3}). With little additional work we can verify that they satisfy the same phenomenology for negative values, trivial zeroes, etc. that we described in \S \ref{negative-values} and \ref{positive-values}.

\subsubsection{Review of Ferraro's Theorem}
We now review a functional identity discovered by Ferraro in \cite{FER1,FER2} (see Theorem \ref{Ferraro-theorem}) which  serves to deduce partial analogues of the functional equation of Riemann's $\zeta$, and also an analogue of Riemann's $\xi$. A precursor of this result has been first obtained by Green and Papanikolas in \cite{GRE&PAP} in the 
case $(X,\infty)$ elliptic curve over $\FF_q$. The first prototype, over the projective line, is in \cite{PEL} but methods of analysis are a little bit different and will not be discussed here.

To read the statement of Ferraro's Theorem \ref{Ferraro-theorem} (we give a simplified version of his results) technical preparation is requested, presented in \S \ref{Drinfeld-shtuka-divisors}. In particular, we need to review the theory of shtuka divisors and functions. The results in this section provide us with exactly $h_A$ effective divisors 
$V$ on $X_{\CC_\infty}$ of degree $g$ the genus of $X$ ($h_A$ is the class number of $A$) called shtuka divisors. They satisfy 
$h^0(V)=1$. For any shtuka divisor $V$ there exists, up to a non-zero scalar factor, a unique shtuka function
$f$ rational over $X_{\CC_\infty}$ such that its divisor satisfies
$$\Div(f)=V^{(1)}-V+\Xi-\infty$$  (see Lemma \ref{drinfeld-lemma}). 
Additionally, to each such shtuka divisor $V$ corresponds a unique dual shtuka divisor $V^*$ (see Lemma \ref{corollary-Weil}), effective of degree $g$, with $h^0(V^*)=1$. In fact $V^*$ is the birational inverse of $V$ in the sense of Weil. Explicitly, by Lemma \ref{corollary-Weil}, there exists up to a scalar factor, a unique rational function $\delta$ over $X_{\CC_\infty}$ such that 
$$\Div(\delta)=V+V^*-2g\infty.$$
We will not fully justify it here, but it is also possible to show that $\delta$ can be chosen in $A\otimes H$ and $f$ in 
$\operatorname{Frac}(A\otimes H)$ with $H$ the Hilbert class field of $A$.

Ferraro's theorem also needs the notion of {\em special function} (introduced in \cite{ANG&NGO&TAV1}). This is a class of elements of $\TT_A$ that extend to meromorphic functions on $X_{\CC_\infty}\setminus\{\infty\}$, depending on the choice of a shtuka function $f$, and that we quickly review in \S \ref{Special-functions}, to play the role of the classical ``gamma factor'' in the corresponding functional equation of Riemann's zeta function. Note that Anderson's twist extends uniquely to an $A\otimes1$-automorphism of $\TT_A$.
Once $f$ is chosen $\omega$ is characterized, up to a factor in $A\otimes1$, by 
$$\omega^{(1)}=f\omega.$$
We finally recall the element $\widetilde{\pi}\in\CC_\infty^\times$, analogue of $2\pi i$, involved in Theorem \ref{Goss-formula-Euler}, and also depending on our choice of $(X,\infty,\sgn)$ and defined up to a multiple in $\FF_q^\times$.
A special case of Ferraro's result is the following: 

\begin{Theorem}[Ferraro]\label{Ferraro-theorem}
For any triple $(X,\infty,\sgn)$ there is a unique shtuka function $f$  associated to a shtuka divisor $V$ together with a non-zero special function $\omega$, and a unique choice of $\delta\in A\otimes H$, such that 
\begin{equation}\label{functional-identity-Ferraro}
\omega^{(1)}Z_{X,A}=\widetilde{\pi}\delta^{(1)}.
\end{equation}
\end{Theorem}
More generally, Ferraro's result provides an identity as above for any partial zeta function $Z_{X,I}$ and the correspondence $I\mapsto V$ is governed by Artin reciprocity law for the extension $H/K$ but for simplicity we will not consider these aspects here. Note that $Z_{X,A}$ and $\omega$ are transcendental functions; the theorem says that they are meromorphic over $X_{\CC_\infty}\setminus\{\infty\}$ but both have essential singularities supported at $\infty$. Their product combines to neutralize the essential singularities and gives rise to a rational function $\delta$ which is however difficult to compute in full generality, just like $f$.

The zeroes of the right-hand side of \eqref{functional-identity-Ferraro} correspond to the degree $2g$ effective divisor $V^{(1)}+{V^*}^{(1)}$. The points in the support of ${V^*}^{(1)}$ can be assimilated 
with the {\em non-trivial zeroes} of $Z_{X,A}$. At once, through \eqref{functional-identity-Ferraro} we immediately see that there is a meromorphic extension of $Z_{X,A}$ to $X_{\CC_\infty}\setminus\{\infty\}$ (of course, we already know, by Lemma \ref{partial-zeta-are-entire}, that it is entire, and in fact this property can be easily deduced from \eqref{functional-identity-Ferraro} and some calculations on divisors). In other words, the function
$\widetilde{\pi}\delta^{(1)}$ can be viewed in this perspective as a function field alternative of Riemann's $\xi$, that can also be characterized by its 
Hadamard's factorization
$$\xi(s)=\prod _{\rho }\left(1-{\frac {s}{\rho }}\right),$$
running over the non-trivial zeroes $\rho$ of $\zeta$ (\footnote{For convergence, the product filters the zeroes through finite subsets are invariant by conjugation.}). Similarly from Corollary \ref{meromorphic-omega} we deduce that 
$Z_{X,A}$ also has the {\em trivial zeroes} $\Xi^{(1)},\Xi^{(2)},\ldots$. Another suggestive analogy is described in \cite[(1)]{FER1}, where classical Dirichlet $L$-series are viewed as functions with Dirichlet characters as variables, a viewpoint that has 
something to share with some remarks in \cite{PEL}. 

\subsubsection{Examples}

In the settings of \S \ref{in-genus-0}, Ferraro's formula \eqref{functional-identity-Ferraro} becomes simpler and more explicit. 
In this case (see \cite{PEL}) we can choose, writing $t:=\theta\otimes1$ and identifying $\theta$ with $1\otimes\theta$, $$\omega=(-\theta)^{\frac{1}{q-1}}\prod_{i\geq1}\Big(1-\frac{t}{\theta^{q^i}}\Big)^{-1}.$$ This is an invertible element of the Tate algebra
$$\TT_{\FF_q[\theta]}=\Big\{\sum_{j\geq0}a_jt^j:a_j\in \FF_q((1/\theta)),a_j\rightarrow0\Big\}.$$

The function $\delta$ reduces here to the constant function $-1$. The formula becomes
$$\omega^{(1)}Z_{\PP^1,\FF_q[\theta]}=-\widetilde{\pi},$$ with $\widetilde{\pi}$ defined by \eqref{Carlitz-pi} and 
where we can write, explicitly,
$$Z_{\PP^1,\FF_q[\theta]}=\sum_{\begin{smallmatrix} a\in A\\ \text{monic}\end{smallmatrix}}\frac{a(t)}{a}.$$
%We also recall again that in \cite{GRE&PAP}, Green and Papanikolas found similar explicit formulas, precursors of Ferraro's work, in the case where $(X,\infty)$ is an elliptic curve over $\FF_q$.  

\section{Ferraro's Theorem}\label{Drinfeld-shtuka-divisors}

In this section we discuss Ferraro's Theorem \ref{Ferraro-theorem}. For convenience of the reader we review the construction of the shtuka and dual shtuka divisors $V,V^*$ and various rational functions that come into play along with them. We also discuss connected themes (motives, special functions) that should not be dissociated from the theory. To give a start, we discuss the following result of Drinfeld (see \cite{AND2,THA-1,FER1} and \cite[\S 4.3 and more precisely Proposition 4.25]{FER1}).

\begin{Lemma}[Drinfeld]\label{drinfeld-lemma}
Choose a triple $(X,\infty,\sgn)$, let $g$ be the genus of $X$ and $h_A$ the class number of $A$. There exist exactly $h_A$ distinct effective divisors $V$ of degree $g$ on 
$X_{\CC_\infty}$ with $h^0(V)=1$ and such that there exists a rational function $f$ over $X_{\CC_\infty}$ with $\Div(f)=V^{(1)}-V+\Xi-\infty$.
\end{Lemma}

In the statement above, any divisor $V$ and any function $f$ as above are called respectively {\em a shtuka divisor} and a {\em shtuka function} for $(X,\infty)$. The divisor $V^{(1)}$ denotes the image of $V$ by the map $x\mapsto x^{(1)}$ over $X_{\CC_\infty}$. 
If $g=0$ then $V=0$. Indeed $\Xi-\infty$ is the divisor of the rational function $f=t-\theta$ in the settings of the Carlitz case \S \ref{in-genus-0}. 

\subsection{Proof of Drinfeld's lemma}\label{construction-shtuka-divisors} Among several available proofs of Lemma \ref{drinfeld-lemma}, we follow Mumford \cite{MUM}. More general statements are contained in \cite{THA-1,FER1} (the underlying problem is, given a degree zero divisor $D$, to find effective divisors $V$ such that $V^{(1)}-V+D$ is principal, and the results proposed in these references give partial responses to this).
We first justify the existence of such divisors in the hypotheses of Lemma \ref{drinfeld-lemma}. Denote by $\mathcal{A}$ the jacobian of $X$ (with respect to $\infty$). It is a principally polarized abelian variety over $\FF_q$ of dimension $g$. Its underlying group law is that of the quotient of the group of divisors of degree zero by the subgroup of its principal divisors. We look at the abelian variety
$$\mathcal{A}_{\CC_\infty}=\mathcal{A}\times_{\operatorname{Spec}(\FF_q)}\operatorname{Spec}(\CC_\infty).$$
We identify it with its $\CC_\infty$-points and we write $\mathcal{A}_{\CC_\infty}=\mathcal{A}(\CC_\infty)$. We also identify 
$X_{\CC_\infty}=X(\CC_\infty)$ etc.
There is a surjective reduction group homomorphism:
$$\operatorname{red}:\mathcal{A}(\CC_\infty)\rightarrow\mathcal{A}(\FF_q^{\text{alg}}).$$
It can be easily defined observing that writing $\mathcal{O}_{\CC_\infty}$ for the unit disk of $\CC_\infty$ (the valuation ring),
$\mathcal{A}(\CC_\infty)=\mathcal{A}(\mathcal{O}_{\CC_\infty})$. In this way, the reduction map is induced by the reduction
map of $\mathcal{O}_{\CC_\infty}$ over the residue field $\FF_q^{\text{alg}}$. The reduction maps are studied in \cite{FER1}
when considering $\mathcal{A}(L)$ with $L$ a finite extension of $K_\infty$, but the properties we review here can be easily deduced from that case.

The kernel of $\operatorname{red}$ has a structure of formal group $\widehat{\mathcal{A}}(\mathfrak{m}_{\CC_\infty})$ and can be identified with an open subgroup of $\mathcal{A}(\CC_\infty)$. It is more precisely
the topological subgroup of $\mathcal{A}(\CC_\infty)$ whose elements are the points $x$ of $\mathcal{A}(\CC_\infty)$
such that $x^{(i)}\rightarrow 0$, the neutral element, as $i$ tends to infinity. It is homeomorphic to $\mathfrak{m}_{\CC_\infty}^g$, the direct sum of $g$ copies of the maximal ideal of $\CC_\infty$. This yields an analogue of \eqref{union-gm} and Proposition \ref{proposition-uvw}, namely that $\mathcal{A}(\CC_\infty)$ can be decomposed in a disjoint union:
\begin{equation}
\mathcal{A}(\CC_\infty)=\bigsqcup_{x\in\mathcal{A}(\FF_q^{\text{alg}})}x+\widehat{\mathcal{A}}(\mathfrak{m}_{\CC_\infty}),
\end{equation}
where we sum for the group law of $\mathcal{A}$. In practice, if we want to compare with \eqref{proposition-uvw}, let us remember that there is no analogue of $\QQ$-lines, but $\mathcal{A}(\FF_q^{\text{alg}})$ corresponds to $(\FF_q^{\text{alg}})^\times$
and $\widehat{\mathcal{A}}(\mathfrak{m}_{\CC_\infty})$ corresponds to $\mathfrak{m}_{\CC_\infty}$. 

Now we observe two facts.
(1) There is a commutative diagram (see Anderson \cite[\S 3.1]{AND2}):
\begin{equation}\label{Lang-torsor}
\begin{tikzcd}
Y(\CC_\infty) \arrow[r, "\beta"] \arrow[d, "\Pi"'] & \mathcal{A}(\CC_\infty) \arrow[d, "1-F"] \\
X(\CC_\infty) \arrow[r, "{[}(\cdot)-\infty{]}"'] & \mathcal{A}(\CC_\infty)
\end{tikzcd}
\end{equation}
with $Y(\CC_\infty):=\operatorname{Spec}(B\otimes_{\FF_q}\CC_\infty)\sqcup\{\infty_i:i=1,\ldots,h_A\}$ ($B$ is the integral closure of $A$ in $H$ and $h_A$ the class number of $A$) with a unramified Galois covering $\Pi$ of $X(\CC_\infty)$ whose group of automorphisms is isomorphic to $G=\operatorname{Gal}(H/K)$, abelian of order $h_A$, and which sends the point $\infty_i$ to $\infty$ for all $i$, the map $P\mapsto[P-\infty]$
is the jacobian map sending a point $P$ to the linear class $[P-\infty]\in \mathcal{A}(\CC_\infty)$ of the divisor $P-\infty$ and $1-F$ is the separable isogeny equal to the
identity minus (the group law of $\mathcal{A}$) the $q$-Frobenius, the kernel of which is $\mathcal{A}(\FF_q)$, isomorphic to the class group $\operatorname{Cl}(A)$ of $A$. The diagram is compatible with Artin's isomorphism $G\rightarrow\operatorname{Cl}(A)$.

(2) Let $g$ be the genus of $X$ and $\mathfrak{S}_g$ the group of permutations of a set with $g$ elements.
Let $X(\CC_\infty)^g/\mathfrak{S}_g$ be the proper variety of dimension $g$ whose points are the effective divisors of degree $g$
\cite[Propositions 3.1 and 3.2]{MIL}.
The map 
\begin{equation}\label{Jg}
[(\cdot)-g\infty]:X(\CC_\infty)^g/\mathfrak{S}_g\rightarrow\mathcal{A}(\CC_\infty)\end{equation}
sending a divisor $D$ to the linear class of $D-g\infty$ is a birational equivalence. Note that if $D=\sum_iD_i$ with $D_i$ effective divisors with degree $1$, then
$[D-g\infty]=\sum_i[D_i-\infty]$.

Combining the facts (1) and (2) we can show the existence of shtuka divisors $V$ as requested, in the following way. Note that $\xi:=J(\Xi)\in\widehat{\mathcal{A}}(\mathfrak{m}_{\CC_\infty})$ (so the reduction of $\Xi$ is $\infty$), hence $\xi^{(i)}\rightarrow0$ in $\mathcal{A}(\CC_\infty)$ as $i$ tends to $\infty$. Therefore the Fredholm series
\begin{equation}\label{defi-v}
v:=\sum_{i\geq0}\xi^{(i)}\end{equation}
converges in $\widehat{\mathcal{A}}(\mathfrak{m}_{\CC_\infty})$ (summing for the group law of $\widehat{\mathcal{A}}(\mathfrak{m}_{\CC_\infty})$). Clearly $v$ belongs to the kernel of $1-F$ so that we can compute:
$$\operatorname{Ker}(1-F)=v+\mathcal{A}(\FF_q).$$
Let $w\in\operatorname{Ker}(1-F)$. By fact (2) there exists $V\in X(\CC_\infty)^g/\mathfrak{S}_g$ with $[V-g\infty]=w$. This is one of the effective divisors that we are looking for, the above construction indeed tells us that there are at least $h_A=|G|$ (class number of $A$) such divisors.

We now conclude the proof by showing that the number of effective divisors $V$ of degree $g$ with the required property is exactly $h_A$. To see this we show that given any such $V$, we have $h^0(V)=1$. We follow Mumford \cite[\S 3]{MUM} including here a proof of a special case of Drinfeld's ``($\chi=0\implies h^0=h^1=0$)-Lemma'' (see also \cite[Lemma 3.3.1]{AND2}). 

Given more generally a divisor $D$, we denote by $\mathcal{L}(D)$ the finite dimension $\CC_\infty$-vector space of rational functions $\varphi$ such that $\Div(\varphi)+D\geq0$. 

Let $f$ be a rational function over $X(\CC_\infty)$ such that
\begin{equation}\label{shtuka-function}
\Div(f)=V^{(1)}-V+\Xi-\infty.
\end{equation}
 Note that, as sets, $\mathcal{L}(D)^{(1)}=\mathcal{L}(D^{(1)})$.
Now define:
$$\mathcal{F}_n:=\mathcal{L}(V+(n-1)\infty),\quad n\in\ZZ.$$
By definition we have that $x\in\mathcal{F}_n^{(1)}$ if and only if $\Div(x)+V^{(1)}+(n-1)\infty\geq0$ (recall that $\infty$ is defined over $\FF_q$). This is equivalent to $\Div(fx)+V-\Xi+n\infty\geq 0$. In other words,
\begin{equation}\label{fundamental-iso}
f\mathcal{F}_n^{(1)}=\mathcal{F}_n(\infty-\Xi).
\end{equation}
But $\mathcal{F}_n(\infty-\Xi)=\mathcal{L}(V-\Xi+n\infty)\subset\mathcal{L}(V+n\infty)=\mathcal{F}_{n+1}$. So we have inclusions
for all $n$ as follows. 
\begin{equation}\label{diagram-F}
\begin{tikzpicture}[scale=1]
  % vertices
  \node (A) at (0,2) {$\mathcal{F}_{n+1}$};    % top
  \node (B) at (1,1) {$f\mathcal{F}^{(1)}_{n}$};    % right
  \node (C) at (0,0) {$\mathcal{F}_{n}\cap f\mathcal{F}_n^{(1)}$};    % bottom
  \node (D) at (-1,1) {$\mathcal{F}_{n}$};   % left
  \node (E) at (0,-1) {$f\mathcal{F}_{n-1}^{(1)}$};   % below C

  % edges of the square
  \draw (A) -- (B);
  \draw (B) -- (C);
  \draw (C) -- (D);
  \draw (D) -- (A);

  % extra edge CE
  \draw (C) -- (E);
\end{tikzpicture}
\end{equation}
Now the bottom edge connecting $\mathcal{F}_{n}\cap f\mathcal{F}_n^{(1)}$ and $f\mathcal{F}_{n-1}^{(1)}$ is an equality,
because the intersection equals $\mathcal{L}(V+(n-1)\infty)\cap\mathcal{L}(V-\Xi+n\infty)=\mathcal{L}(V-\Xi+(n-1)\infty)=f\mathcal{F}_{n-1}^{(1)}$.
For example, if $X=\PP^1$, $\infty=(1:0)$ when we have seen that $V=0$, we also have, by direct computation, 
$$\mathcal{F}_n=\operatorname{Vect}_{\CC_\infty}(1,t,\ldots,t^{n-1}),\quad n\geq 1.$$
The rest of the proof will allow us to also see that for these spaces, more appropriate bases exist (see Corollary \ref{corollary-better-bases}).
We now set, for a divisor $D$:
\begin{eqnarray*}
h^0(D)&:=&\dim_{\CC_\infty}(\mathcal{L}(D))\\
h^1(D)&:=&\dim_{\CC_\infty}(\mathcal{L}(K_X-D))=\dim_{\CC_\infty}(\Omega^1_X(-D))\\
\chi(D)&:=&h^0(D)-h^1(D)\\
&=&\deg(D)-g+1,
\end{eqnarray*}
where $K_X$ is a canonical divisor for $X$. The last identity is Riemann-Roch's theorem. Note that if $D\leq D'$ then $\mathcal{L}(D)\subset\mathcal{L}(D')$ so that
$h^0(D)\leq h^0(D')$, and similarly, $h^1(D)\geq h^1(D')$. We want to show $h^0(V-\infty)=h^1(V-\infty)=0$ (the ($\chi=0\implies h^0=h^1=0$)-Lemma).

We immediately see $\chi(\mathcal{F}_n):=\chi(V+(n-1)\infty)=n$ for all $n$. We can find $n$ with $h^0(V+(n-1)\infty)\neq0$.
For example $n=1$ works: $\Div(1)+V\geq 0$ so that $1\in\mathcal{L}(V)$. If $n\ll 0$ however, $h^0(V+(n-1)\infty)=0$. Indeed
if $D$ has negative degree and $f\in\mathcal{L}(D)\setminus\{0\}$, $\Div(f)+D\geq0$ and $\deg(f)+\deg(D)=\deg(D)\geq0$ which is impossible. There exists a smallest integer $n_0\leq 1$ such that $h^0(V+(n_0-1)\infty)\neq0$. 

We show that $n_0=1$. Choose $s\in\mathcal{F}_{n_0}\setminus\{0\}$. Define $s_n=fs_{n-1}^{(1)}\in\mathcal{F}_{n_0+n}$
for all $n\geq 0$ (we proved that $f\mathcal{F}^{(1)}_n\subset\mathcal{F}_{n+1}$). So we have the functions
$$s=s_0,\quad s_1=fs_0^{(1)},\quad s_2=ff^{(1)}s_0^{(2)},\ldots,s_n=ff^{(1)}\cdots f^{(n-1)}s_0^{(n)}.$$
Claim: $s_n\in\mathcal{F}_{n_0+n}\setminus\mathcal{F}_{n_0+n-1}$. Otherwise $s_n\in\mathcal{F}_{n_0+n-1}$
but $s_n=fs_{n-1}^{(1)}\in\mathcal{F}_{n_0+n-1}^{(1)}$ and therefore $s_n=fs_{n-1}^{(1)}\in\mathcal{F}_{n_0+n-1}\cap f\mathcal{F}_{n_0+n-1}^{(1)}=f\mathcal{F}_{n_0+n-2}^{(1)}$. This implies $s_{n-1}=fs_{n-2}^{(1)}\in\mathcal{F}_{n_0+n-2}$
and inductively we get $s_{n-j+1}\in\mathcal{F}_{n_0+n-j}$ for all $j$ so that $s_0\in\mathcal{F}_{n_0-1}$. This latter space is zero by hypothesis and we reach a contradiction because $s_0$ is assumed to be non-zero, justifying that for all $n$, 
$s_n\in\mathcal{F}_{n_0+n}\setminus\mathcal{F}_{n_0+n-1}$ for all $n$.

As a consequence $s_0,s_1,\ldots,s_n$ are linearly independent in $\mathcal{F}_{n_0+n}$ for all $n$. For $n\gg0$ we have
$$n+1\leq h^0(V+(n+n_0-1)\infty)=\chi(V+(n+n_0-1)\infty)=n_0+n$$ (if $n$ is large enough, $h^1(V+(n+n_0-1)\infty)=0$). Hence
$n_0=1$ and therefore, $h^0(V-\infty)=h^1(V-\infty)=0$. But then, $h^0(V)=1$ as we wanted, indeed $\chi(\mathcal{F}_1)=h^0(V)-h^1(V)=1$ but from $0=h^0(V-\infty)=h^1(V-\infty)\geq h^1(V)$ we deduce $h^1(V)=0$ implying that $h^0(V)=1$. So $\mathcal{F}_1=\CC_\infty$ and $V$ is the unique effective divisor in its linear class (assuming by contradiction that $W$ is another such effective divisor with $V\sim W$ (the symbol $\sim$ denotes linear equivalence of divisors) then there exists $h$ rational function such that $\Div(h)=W-V$ so that $\Div(h)+V\geq0$, $h\in\mathcal{F}_1$ and $h\in\CC_\infty$, so that $V=W$).

We record:

\begin{Corollary}\label{corollary-better-bases}
Consider an integer $n\geq 0$. The $n+1$ functions $ff^{(1)}\cdots f^{(i-1)}$ for $i=0,\ldots,n$ determine a $\CC_\infty$-basis of $\mathcal{F}_{n+1}$.
\end{Corollary}

In particular, $$1,t-\theta,\ldots,(t-\theta)(t-\theta^q)\cdots(t-\theta^{q^{n-1}})$$
is another $\CC_\infty$-basis of $\mathcal{F}_{n+1}=\mathcal{L}(n\infty)$ in the Carlitz case $X=\PP^1$ and $\infty=(1:0)$ of \S \ref{in-genus-0}.
 
\subsection{Dual shtuka divisors}

The results and proofs of \S \ref{construction-shtuka-divisors} can be modified in various ways (we refrained to give a general statement). One of them is given by the following variant of Lemma \ref{drinfeld-lemma}, where $h_A$ is the class number of $A$. In the statement, $\sim$ denotes linear equivalence.

\begin{Lemma}\label{lemma-dual-shtuka}
In the settings of Lemma \ref{drinfeld-lemma}, there exist exactly $h_A$ effective divisors $V^*$ of degree $g$ such that
$$V^*-(V^*)^{(1)}+\Xi-\infty\sim0.$$ Any such divisor satisfies $h^0(V^*)=1$.
\end{Lemma}

\begin{proof}[Sketch of proof] We can modify the proof of Lemma \ref{drinfeld-lemma} (($\chi=0\implies h^0=h^1=0$)-Lemma) but, to add some turbulences, we can also use that lemma itself to reach what we precognize.
We choose a shtuka divisor $V$. We modify the flag of vector spaces $\mathcal{F}_n$
by setting $$\mathcal{G}_n:=\mathcal{L}(-V+(2g-1+n)\infty)$$ so that $\mathcal{G}_0=\mathcal{L}(-V+(2g-1))$ and $\mathcal{G}_n=\mathcal{G}_0(n\infty)$. Again we have $\mathcal{G}_n\subset\mathcal{G}_{n+1}$ for all $n$ and $x\in\mathcal{G}_n$ is equivalent, with $f$ shtuka function such that $\Div(f)=V^{(1)}-V+\Xi-\infty$, to $\Div(fx)-\Div(f)-V+(2g-1+n)\infty\geq0$, and replacing 
$\Div(f)$ with $V^{(1)}-V+\Xi-\infty$ we have $x\in\mathcal{G}_n$ if and only if $fx\in\mathcal{L}(-V^{(1)}-\Xi+(2g+n)\infty)=\mathcal{G}_n^{(1)}(\infty-\Xi)\subset
\mathcal{G}_{n+1}^{(1)}$. 
In other words,
\begin{equation}\label{fundamental-iso-dual}
f\mathcal{G}_n=\mathcal{G}^{(1)}_n(\infty-\Xi).
\end{equation}
Hence we get the diagram of inclusions, where the bottom edge is an equality.
\[
\begin{tikzpicture}[scale=1]
  % vertices
  \node (A) at (0,2) {$\mathcal{G}_{n+1}$};    % top
  \node (B) at (1,1) {$f^{(-1)}\mathcal{G}^{(-1)}_{n}$};    % right
  \node (C) at (0,0) {$\mathcal{G}_{n}\cap f^{(-1)}\mathcal{G}_n^{(-1)}$};    % bottom
  \node (D) at (-1,1) {$\mathcal{G}_{n}$};   % left
  \node (E) at (0,-1) {$f^{(-1)}\mathcal{G}_{n-1}^{(-1)}$};   % below C

  % edges of the square
  \draw (A) -- (B);
  \draw (B) -- (C);
  \draw (C) -- (D);
  \draw (D) -- (A);

  % extra edge CE
  \draw (C) -- (E);
\end{tikzpicture}
\]
Just as in the proof of Lemma \ref{construction-shtuka-divisors} we see that $\chi(\mathcal{G}_n)=n$ for all $n$,
following the same ideas we deduce easily that $h^0(-V+(2g-1)\infty)=h^1(-V+(2g-1)\infty)=0$. In particular we get
$h^0(-V+2g\infty)=1$. We have reached:

\begin{Lemma}\label{corollary-Weil}
Given a shtuka divisor $V$ there exists a unique effective divisor $V^*$ of degree $g$ such that
$$V^*+V-2g\infty\sim0$$ and up to a non-zero scalar, a unique rational function $\delta$ such that
$(\delta)=V+V^*-2g\infty$.
\end{Lemma}
We also have $h^0(V^*)=1$. Additionally, a quick check shows that we have isomorphisms induced by the multiplication by $\delta$
\begin{equation}\label{isomorphisms-delta}
\mathcal{G}_n\xrightarrow{x\mapsto x\delta}\mathcal{L}(V^*+(n-1)\infty),\quad n\in\ZZ.
\end{equation}
If we set $\mathcal{F}^*_n:=\mathcal{L}((V^*)^{(1)}+(n-1)\infty)$, we obtain a flag of vector spaces of rational functions over $X(\CC_\infty)$ analogous to the previously mentioned flag $\mathcal{F}_n$.
The existence of divisors like $V^*$ (dual shtuka divisors) is also guaranteed by the surjectivity of $[(\cdot)-g\infty]$ by taking inverse images of $-v\in\widehat{\mathcal{A}}(\mathfrak{m}_{\CC_\infty})$ with $v$ defined in 
\eqref{defi-v}. This goes back to the interpretation of $V^*$ as the opposite of $V$ for Weil's birational group laws.
\end{proof}

Let $f^*$ be any rational function such that $\Div(f^*)=V^*-(V^*)^{(1)}+\Xi-\infty$ with $V^*$ any divisor as in Lemma \ref{drinfeld-lemma}. We can construct it explicitly by choosing $f$ with divisor $V^{(1)}-V+\Xi-\infty$, then setting $\tilde{f}=\hat{f}^{-1}$ where $$\hat{f}=\frac{\delta}{(f\delta)^{(-1)}},$$
which satisfies
$$\Div(\tilde{f})=(V^*)^{(-1)}-V^*+\Xi^{(-1)}-\infty,$$
and then taking
$$f^*:=\tilde{f}^{(1)}=\frac{f\delta}{\delta^{(1)}},$$
or any non-zero multiple of it.

\begin{Corollary}\label{corollary-better-bases-dual}
Consider an integer $n\geq 0$. The functions ${f^*}^{(-1)}{f^*}^{(-2)}\cdots {f^*}^{(-i)}$ for $i=0,\ldots,n$ determine a $\CC_\infty$-basis of $\mathcal{G}_{n+1}$. Equivalently the functions ${f^*}{f^*}^{(-1)}\cdots {f^*}^{(1-i)}$ for $i=0,\ldots,n$ determine a $\CC_\infty$-basis of $\mathcal{F}^*_{n+1}$.
\end{Corollary}
So if $i=0$, the basis is $\{1\}$. For example, $$1,t-\theta^{\frac{1}{q}},\ldots,(t-\theta^{\frac{1}{q}})(t-\theta^{\frac{1}{q^2}})\cdots(t-\theta^{\frac{1}{q^n}})$$
is a $\CC_\infty$-basis of $\mathcal{G}_{n+1}=\mathcal{L}(n\infty)$ in the case $X=\PP^1$ and $\infty=(1:0)$.

\begin{Corollary}
The points $\infty,\Xi$ do not belong to the supports of $V$ and $V^*$.
\end{Corollary}
\begin{proof}
It follows from the property $h^0(V)=h^0(V^*)=1$, see also \cite[Corollary 0.3.3]{THA-1}.
\end{proof}
The shtuka functions $f$, the dual shtuka functions $f^*$, and the functions $\delta$, determined up to multiplication by a scalar factor, are difficult to compute in general, for arbitrary choices of $X$. For $X$ hyperelliptic, explicit formulas exist, see \cite{FER1}.

\subsection{Anderson motives}

We set 
\begin{eqnarray*}
\mathcal{M}&:=& H^0(X_{\CC_\infty}\setminus\{\infty\},\mathcal{O}_{X_{\CC_\infty}}(V))=\bigcup_{n\geq1}\mathcal{F}_n\\ \mathcal{M}^*&:=& H^0(X_{\CC_\infty}\setminus\{\infty\},\mathcal{O}_{X_{\CC_\infty}}(V^*)^{(1)})=\bigcup_{n\geq1}\mathcal{F}_n^*.
\end{eqnarray*}
These are infinite dimension $\CC_\infty$-vector spaces of rational functions over $X_{\CC_\infty}$, but they have considerably more structure
and this is what we want to discuss now. We denote by $\CC_\infty[\tau]$ the non-commutative $\CC_\infty$-algebra of polynomials $a_0+\cdots+a_n\tau^n$ with $a_0,\ldots,a_n\in\CC_\infty$ with the commutation rule $\tau c=c^q\tau$
for $c\in\CC_\infty$. Let $f$ be a shtuka function for $(X,\infty)$. By Corollary \ref{corollary-better-bases}, the map $\tau^i\mapsto ff^{(1)}\cdots f^{(i-1)}$
extends to a $\CC_\infty$-isomorphism
\begin{equation}\label{iota-iso}
\CC_\infty[\tau]\xrightarrow{\iota}\mathcal{M}.
\end{equation}
The diagrams \eqref{diagram-F} guarantee that the above morphism equips $\mathcal{M}$ with a structure of left $\CC_\infty[\tau]$-module, free of rank one generated by $1\in\mathcal{M}$, depending on the choice of the shtuka function $f$. The action of $\tau$ on an element $x\in\mathcal{M}$ is defined by $\tau x:=fx^{(1)}$, where the map $x\mapsto x^{(1)}$ is the Anderson twist discussed in \S \ref{partial-zeta-functions}, which is $A\otimes1$-linear. In a similar vein, defining this time $\CC_\infty[\tau^{-1}]$
to be the $\CC_\infty$-algebra of polynomials $a_0+\cdots+a_n\tau^{-n}$ with $a_0,\ldots,a_n\in\CC_\infty$ with the commutation rule $\tau^{-1} c=c^\frac{1}{q}\tau^{-1}$, the choice of a dual shtuka function $f^*$ determines a structure of left
$\CC_\infty[\tau^{-1}]$-module on $\mathcal{M}^*$, which is free of rank one generated by $1$. The $\CC_\infty$-isomorphism
is defined by Corollary \ref{corollary-better-bases-dual} by sending $\tau^{-i}$ to ${f^*}{f^*}^{(-1)}\cdots {f^*}^{(1-i)}$, 
and if $x\in\mathcal{M}^*$ we set $\tau^{-1}x=f^*x^{(-1)}$ (as already noticed, since $\CC_\infty$ is perfect, the Anderson twist is an automorphism).

More is true.

\begin{Lemma}\label{anderson-motives}
the vector space $\mathcal{M}$ carries a structure of right $A$-module commuting with the left $\CC_\infty[\tau]$-module structure.
So it can be seen as a left $(A\otimes_{\FF_q}\CC_\infty)[\tau]$-module where $\tau$ commutes with all the elements $a\otimes1$ with $a\in A$.
\end{Lemma}

\begin{proof}
We explain how to define a right action of $A$ on $\mathcal{M}$. Note that for all $a\in A$, $\deg(a)=d$,
$a\otimes1\in\mathcal{L}(d\infty)\subset\mathcal{L}(V+(d+1-1)\infty)=\mathcal{F}_{d+1}$.
Applying Corollary \ref{corollary-better-bases-dual} we can expand, in unique way,
\begin{equation}\label{decomposition-a}
a\otimes1=\sum_{i=0}^{d}[a]_iff^{(1)}\cdots f^{(i-1)}\in\mathcal{F}_{d+1},
\end{equation}
where the coefficients $[a]_0,\ldots,[a]_d$ are in $\CC_\infty$ with $d=\deg(a)$. Note that $f(\Xi)=0$ while $f^{(i)}$ has no pole in $\Xi$ for $i>0$, so that
$$[a]_0=(a\otimes1)(\Xi)=\Xi(a)=a\in\CC_\infty.$$ Also note that $[a]_d\neq0$. Indeed otherwise, observe the equality of divisors
\begin{equation}\label{interated-shtuka}
ff^{(1)}\cdots f^{(i-1)}=V^{(i)}-V+\sum_{j=0}^{i-1}\Xi^{(j)}-i\infty,\quad i\geq0.
\end{equation}
Assuming by contradiction that $[a]_d=0$ for some $a\in A$ of degree $d$, we would have $a\otimes1\in\mathcal{L}((d-1)\infty)$ against the hypothesis on the degree.
For all $n$,
$$(a\otimes1)\mathcal{F}_{n}\subset\mathcal{F}_{d+n}.$$ the multiplication by $a\otimes1$ is therefore an injective 
$\CC_\infty$-endomorphism of $\mathcal{M}$. Coming back to the isomorphism $\iota$ in \eqref{iota-iso}, multiplying by $a\otimes1$ corresponds to the action of an injective endomorphism $\phi_a$ of $\CC_\infty[\tau]$ that has the property that 
for all $n$, it sends $\mathcal{F}_n$ to $\mathcal{F}_{d+n}$. 
Note that for $b\in A$, $(b\otimes1)^{(i)}=b\otimes1$ for all $i$ (the equality of endomorphisms $\tau(b\otimes1)=(b\otimes1)\tau$ follows from the associativity of $\CC_\infty[\tau]$). Therefore, if $b\otimes1=\sum_{j=0}^e[b]_jf\cdots f^{(j-1)}$ (with $e=\deg(b)$), for all $i$,
$$b\otimes1=(b\otimes1)^{(i)}=\sum_{j=0}^{e}[b]_j^{q^i}f^{(i)}f^{(i+1)}\cdots f^{(i+j-1)}.$$ 
Consider $s_i=ff^{(1)}\cdots f^{(i-1)}$ an element of the basis considered in Corollary \ref{corollary-better-bases} (choosing $s_0=1$). 
We have 
\begin{equation}\label{si}
s_i(b\otimes1)=s_i(b\otimes1)^{(i)}=\sum_{j=0}^{e}[b]_j^{q^i}ff^{(1)}\cdots f^{(i-1)}f^{(i)}f^{(i+1)}\cdots f^{(i+j-1)}
=\sum_{i=0}^{d}[b]_i^{q^j}s_{i+j}.\end{equation}

Define, for all $a\in A$,
\begin{equation}\label{drinfeld-module}
\phi_a:=\sum_{i=0}^{\deg(a)}[a]_i\tau^i\in\CC_\infty[\tau],
\end{equation} for coefficients $[a]_i\in\CC_\infty$. By \eqref{si} and uniqueness,
$\phi_a\phi_b=\phi_{ab}$.
Hence, for $a,b\in A$, $\phi_a$ and $\phi_b$ commute: $$\phi_{a}\phi_b=\phi_{ab}=\phi_{ba}=\phi_b\phi_a.$$
Identifying $\mathcal{M}$ with $\CC_\infty[\tau]$ as above via $\iota$, the left multiplication by the rational function
$a\otimes1$ corresponds to the right multiplication by $\phi_a$. 
\end{proof}

The ring $A\otimes_{\FF_q}\CC_\infty$ is a Dedekind domain. One shows that $\mathcal{M}$ is a projective $A\otimes_{\FF_q}\CC_\infty$-module of rank one. Following \cite[Definition 2.3.1]{HAR&JUS}, it is an {\em effective $A$-motive} of rank one. 
Indeed note that by \eqref{fundamental-iso}, the multiplication by $f$ induces an injective $A\otimes_{\FF_q}\CC_\infty$-module map
$$\tau^*\mathcal{M}\xrightarrow{\tau_\mathcal{M}}\mathcal{M}$$
where $\tau^*\mathcal{M}=\cup_n\mathcal{F}_{n}^{(1)}=\mathcal{M}^{(1)}$ seen as an $A\otimes_{\FF_q}\CC_\infty$-module (note that for all $n$, $\mathcal{F}_n(\infty-\Xi)\subset\mathcal{F}_{n+1}$). The kernel of the multiplication map $A\otimes_{\FF_q}\CC_\infty\rightarrow\CC_\infty$ defined by $a\otimes c\mapsto ac$ is the maximal ideal $\mathfrak{j}$ generated by the 
elements $a\otimes1-1\otimes a$, $a\in A$ and it is easy to see that the above injective map gives rise, inverting $\mathfrak{j}$, to an isomorphism $\tau_\mathcal{M}:\tau^*\mathcal{M}[\mathfrak{j}^{-1}]\rightarrow\mathcal{M}[\mathfrak{j}^{-1}]$.
This construction obviously depends on the choice of the shtuka function $f$. 

\begin{Remark}{\em 
Similarly, $\mathcal{M}^*$ has a left $(A\otimes_{\FF_q}\CC_\infty)[\tau^{-1}]$-module structure. For the proof, everything formally proceeds in the same way as for $\mathcal{M}$ (this is very much in conformity with Mumford's philosophy in \cite{MUM}). From Corollary \ref{corollary-better-bases-dual}, for every $a\in A$, 
$$a\otimes1=\sum_{i=0}^{d}[a]^*_i{f^*}{f^*}^{(-1)}\cdots {f^*}^{(1-i)}\in\mathcal{F}^*_d,$$
where the coefficients $[a]^*_0,\ldots,[a]^*_d$ are in $\CC_\infty$. Note that $f^*(\Xi)=0$, so that
$[a]^*_0=a\in\CC_\infty.$ Again we get pairwise commuting operators
$\phi^*_a\in\CC_\infty[\tau^{-1}]$, with $a\in A$, with the properties we expect. In particular, $\mathcal{M}^*$ is an {\em effective dual $A$-motive}, or {\em co-motive}, in the sense of \cite[Definition 2.4.1]{HAR&JUS} or \cite[Definition 3.6]{GAZ&MAU}. To see this define $\mathcal{N}:=\cup_{n\geq1}\mathcal{G}_n$, another infinity dimension sub-vector space of the space of rational functions over $X_{\CC_\infty}$. By \eqref{isomorphisms-delta} and the above discussion, it is a $A\otimes_{\FF_q}\CC_\infty$-module isomorphic to $\mathcal{M}^*$, and it is projective. By \eqref{fundamental-iso-dual} the multiplication by $f$ this time determines an injective map 
$$\mathcal{N}\xrightarrow{\tau_\mathcal{N}}\tau^*\mathcal{N},$$ and the conditions requested are satisfied. By \eqref{isomorphisms-delta} again, we find that $\mathcal{M}^*$ carries a structure of dual $A$-motive.
}
\end{Remark}

The proof above motivates the following:

\begin{Definition}\label{definition-Drinfeld}{\em A {\em Drinfeld $A$-module of rank one} is an injective $\FF_q$-algebra homomorphism
$$\phi: A\rightarrow\CC_\infty[\tau]$$
such that, for all $a\in A$, $\phi_a:=\phi(a)$ admits an expansion (\ref{drinfeld-module}) with $[a]_0=a$ and $[a]_{\deg(a)}\neq 0$.}
\end{Definition}

The proof of Lemma \ref{anderson-motives} unveils a correspondence, that will not be discussed here, between data composed by triples $(X,\infty,\sgn)$
and {\em sign-normalized Drinfeld modules of rank one} (we have not discussed much sign functions in the above). It is a Drinfeld module of rank one
$\psi:A\rightarrow B[\tau]$ ($B$ is the ring of integers of $H$) such that for all $d\geq 0$ and for all $a\in A$ of degree $d$, $\psi_a=[a]_0+[a]_1\tau+\cdots+[a]_d\tau^d$ with $[a]_0=a$ and $[a]_d=\sgn(a)$. By \cite[Theorem 7.2.15]{GOS}, every
Drinfeld $A$-module of rank one is isomorphic to a sign-normalized Drinfeld module of rank one (as in Definition \ref{definition-Drinfeld}). Two Drinfeld $A$-modules $\phi,\psi$ are {\em isomorphic} (over $\CC_\infty$) if there exists $c\in\CC_\infty^\times$ such that for all $a\in A$, the following identity in $\CC_\infty[\tau]$ holds: $$\phi_ac=c\psi_a.$$ Additionally, distinct sign-normalized Drinfeld modules of rank one are not isomorphic, and by \cite[Proposition 7.2.17]{GOS}, there are exactly $h_A$ (the class number of $A$) isomorphism classes of rank one Drinfeld modules. These properties we mentioned here are not difficult to prove taking into account the diagram
\eqref{Lang-torsor}; {\em shtuka correspondence} \cite[\S 6.2]{GOS} describes a bijective correspondence between shtuka divisors and classes of isomorphism of Drinfeld $A$-modules. From this, one can also deduce that a Drinfeld $A$-module $\phi$ of rank one has all its shtuka functions
proportional to one $f$ which satisfies 
$$f\in\operatorname{Frac}(A\otimes_{\FF_q}H)\subset\operatorname{Frac}(A\otimes_{\FF_q}K_\infty).$$ 
See \cite[Chapter 7]{GOS} for the complete theory.

\subsection{Special functions}\label{Special-functions} This is yet another crucial part in the structures we are describing. Special functions were introduced,
along with the terminology, in \cite[\S 3.2]{ANG&NGO&TAV1}. Choose a shtuka function $f$ for $(X,\infty)$.
\begin{Definition}\label{definition-Special-functions}{\em A {\em special function} $\omega$ is an element of $\TT_A$ such that 
\begin{equation}\label{tau-difference-equation}
\omega^{(1)}=f\omega.\end{equation}}\end{Definition} 
The set of special functions is an $A$-module denoted by $\mathfrak{Sf}(\phi)$. More precisely, the multiplication of the special function $\omega$ by $a$ is $(a\otimes1)\omega$ (remember that $(a\otimes1)^{(1)}=a\otimes1\in\TT_A$). 
Choose a Drinfeld $A$-module of rank one $\phi$ with shtuka function $f$. 
Using \eqref{decomposition-a}
one proves that equivalently, 
$$\mathfrak{Sf}(\phi)=\{\omega\in\TT_A:\phi_a(\omega)=(a\otimes1)\omega\},$$
where $\phi_a$ denotes here the $A\otimes1$-linear extension of the operator $\phi_a\in\CC_\infty[\tau]$ to $\TT_a$, see
\cite[Lemma 3.6]{ANG&NGO&TAV1}. Moreover, the structure of this $A$-module is satisfactorily described in \cite[Theorem 3.11]{GAZ&MAU}. We deduce that $\mathfrak{Sf}(\phi)$ is isomorphic to an invertible ideal of $A$. In particular,
it is not the zero module. Hence:

\begin{Corollary}\label{meromorphic-omega}
The $A$-module $\mathfrak{Sf}(\phi)$, non-trivial, is contained in the field of meromorphic functions over $X(\CC_\infty)\setminus\{\infty\}$. Any non-zero element of this module, seen as a meromorphic function,
has a simple pole at $\Xi^{(i)}$ for all $i\geq 0$.
\end{Corollary}

\begin{proof}
By \cite[Theorem 3.11]{GAZ&MAU}, $\mathfrak{Sf}(\phi)$ is non-trivial, it contains non-constant functions
over $\boldsymbol{D}_X$ the affinoid subset defined in \eqref{DX}. Pick one of them, call it $\omega$. 
Now set, for $\rho\in|\CC_\infty^\times|$, 
$$\boldsymbol{D}_X(\rho):=\{\chi\in X(\CC_\infty)\setminus\{\infty\}:\{|(a\otimes 1)(\chi)|\rho^{-\deg(a)}:a\in A\}\text{ is bounded}\}.$$
Of course $\boldsymbol{D}_X=\boldsymbol{D}_X(1)\subset\boldsymbol{D}_X(\rho)$ if $\rho\geq1$. 
We have $X(\CC_\infty)\setminus\{\infty\}=\cup_\rho\boldsymbol{D}_X(\rho)$, so that we have determined a partial order 
over $X(\CC_\infty)$ by means of the map $\mu:X(\CC_\infty)\setminus\{\infty\}\rightarrow|\CC_\infty^\times|$
defined by $$\mu(\chi):=\inf\{\rho:\chi\in\boldsymbol{D}_X(\rho)\}.$$ In a similar way we can construct totally ordered 
subsets of $X(\CC_\infty)\setminus\{\infty\}$ that are isometric with $|\CC_\infty^\times|$ but we will not need this here.

Since $\omega\in\TT_A$ we can choose a Banach basis $(a_i\otimes1)_{i\geq0}\subset A\otimes1$ of $\TT_A$ and expand 
\begin{equation} \label{series-omega}
\omega=\sum_{i\geq0}c_ia_i,\quad c_i\in\CC_\infty.
\end{equation} Then $\lim_{i\rightarrow\infty}c_i=0$.
The fact that 
$\omega$ satisfies $\phi_a(\omega)=(a\otimes1)\omega$ for a non-constant $a\in A$ implies that the series \eqref{series-omega} is {\em overconvergent}, that is, 
there exists $\rho>1$ such that $c_i\rho^{\deg(a_i)}\rightarrow0$. Therefore there exists $\rho>1$ such that 
$\omega$ extends to a rigid analytic function over $\boldsymbol{D}_X(\rho)$. At once $\omega^{(1)}$ is rigid analytic over
 $\tau(\boldsymbol{D}_X(\rho))=\boldsymbol{D}_X(\rho^q)\supset \boldsymbol{D}_X(\rho)$ and by \eqref{tau-difference-equation}, $\omega=\frac{\omega^{(1)}}{f}$ extends to a meromorphic function over $\boldsymbol{D}_X(\rho^q)$. Similarly,
 $$\omega=\frac{\omega^{(n)}}{ff^{(1)}\cdots f^{(n-1)}}$$ extends to a meromorphic function over $\boldsymbol{D}_X(\rho^{q^n})$ for all $n\geq0$. We conclude observing that $$\bigcup_n(\boldsymbol{D}_X(\rho^{q^n}))=X(\CC_\infty)\setminus\{\infty\}.$$
 In particular, any non-zero element $\omega\in\mathfrak{Sf}(\phi)$, has simple poles at the points $\Xi,\Xi^{(1)},\ldots$.
\end{proof}

\begin{Remark}{\em The terminology ``overconvergence’’ has been employed, in the above proof, for the following reasons. In an only apparently separate topic Dwork, \cite[Discussion around (5)]{DWO}, constructs a certain overconvergent $p$-adic function $\theta(t)$ that, after analysis of the proof structure, appears to play a role which is analogous to our special functions (that however live on the $\infty$-adic side). In order to maintain the present work self-contained we will not investigate the relationship between Weil’s conjectures, Dwork’s proof, and our theory, but there are underlying deep connections. The author plans to discuss these aspects in another work.}\end{Remark}

\subsection{Partial sums of partial zeta functions}
In this subsection we follow the arguments of \cite{CHU&NGO&PEL}. 
Define $0=j_0<j_1<\cdots$ where for all $d$, $j_d$ is minimal with the property that $\mathcal{L}(j_d\infty)=d+1$. By Riemann-Roch theorem
we have $j_d=d+g$ for $d$ large enough, and there are $g$ Weierstrass gaps in this sequence. Taking $a_0=1$ and, for all $d$, an element $a_{d}\in A$ of degree $j_d$, we have a basis $(a_0,\ldots,a_d)$ of $\mathcal{L}(j_d\infty)$. We are going to see that 
the ``elementary bricks'' of the partial sums of $Z_{X,A}$ determine an alternative basis for $\mathcal{L}(j_d\infty)$.
We define the rational functions
\begin{equation*}
\mathcal{S}_{j_m}:=\sum_{\begin{smallmatrix}a\in A\\ \text{monic}\\ \deg(a)=j_m\end{smallmatrix}}a\otimes a^{-1}\in A\otimes K,
\quad\mathcal{S}_{\leq j_m}:=\sum_{d=0}^m\mathcal{S}_{j_d}.
\end{equation*}
We have
$$Z_{X,A}=\sum_{m\geq0}\mathcal{S}_{j_m},$$ converging for the Gauss norm.
The next result is proved in \cite[\S 2.2]{CHU&NGO&PEL}, compare with \cite[Propositions 6.3, 6.4]{GRE&PAP}. The proof relies on Riemann-Roch theorem and induction.

\begin{Proposition}\label{proposition-new-bases}
The following properties hold.
\begin{enumerate}
\item For all $d\geq 0$, $(\mathcal{S}_{j_0},\ldots,\mathcal{S}_{j_d})$ is a basis of $\mathcal{L}(j_d\infty)$.
\item For all $m\geq 1$, $$\mathcal{S}_{j_m}^{(1)}=\mathcal{S}_{\leq j_m}\cdot\sum_{\begin{smallmatrix}
a\in A \\ \text{monic} \\ \deg(a)=j_m
\end{smallmatrix}}\frac{1}{a^{q-1}}.$$
\end{enumerate}
\end{Proposition}
As a consequence, we derive a good understanding of the divisor of the rational functions $\mathcal{S}_{\leq j_m}$.
\begin{Lemma}\label{divisor-partial-sums}
For all $m$ large enough there exists a degree $g$ effective divisor $V^*_m$ on $X(\CC_\infty)\setminus\{0\}$ such that 
\begin{equation}\label{divisor-of-Sj}
\Big(\mathcal{S}_{\leq j_m}\Big)=(V^*_m)^{(1)}+\sum_{i=1}^m\Xi^{(i)}-j_m\infty.
\end{equation}
Moreover, $[V^*_m-g\infty]=v^{(m)}-v$ where $v$ is the point defined in \eqref{defi-v}. 
\end{Lemma}

\begin{proof}
Observe that for all $m$, $\mathcal{S}_{j_m},\mathcal{S}_{\leq j_m}\in A\otimes_{\FF_q} K$. The order of the pole at infinity is exactly $j_m$ by the conjunction of the two parts of Proposition \ref{proposition-new-bases} and a variant of Lemma \ref{lemma-partial-negative} implies that for $m\geq1$, $\mathcal{S}_{j_m}$ also vanishes at the distinct points $\Xi,\ldots,\Xi^{(m-1)}$ (see \cite[Lemma 2.1]{CHU&NGO&PEL}). By the second part of Proposition \ref{proposition-new-bases} we therefore get that $\mathcal{S}_{\leq j_m}$ vanishes at the points
$\Xi^{(1)},\ldots,\Xi^{(m)}$ for all $m\geq1$. Since $j_m=m+g$ for $m$ large, this implies that there exists, for every $m$ large,
an effective divisor $V^*_m$ of degree $g$ such that \eqref{divisor-of-Sj} holds.
Compute, when $j_m=m+g$ and writing $v^*_m$ in place of $[V^*_m-g\infty]\in\mathcal{A}(\CC_\infty)$, for $m$ large,
$$0=\Big[V^*_m+\sum_{i=0}^{m-1}\Xi^{(i)}-j_m\infty\Big]=v^*_m+\Big[\sum_{i=0}^{m-1}\Xi^{(i)}-m\infty\Big]=v^*_m+\sum_{i=0}^{m-1}\xi^{(i)}=v^*_m+v-v^{(m)}.$$
We deduce $$v^*_m=v^{(m)}-v,\quad m\gg0.$$
\end{proof}
The topological space underlying \eqref{Jg} is complete and the jacobian map $[(\cdot)-g\infty]$ is continuous. Since $v,v^{(m)}\in\widehat{A}(\mathfrak{m}_{\CC_\infty})$ we get
$[V^*_m-g\infty]=v^{(m)}-v\rightarrow -v=[V^*-g\infty]$ as $m\rightarrow\infty$. Therefore, over a strictly increasing subsequence $(m_i)_i$ and in the {\em compact topology} of $(X_{\CC_\infty}^g/\mathfrak{S}_g)(L)$ for a suitably chosen finite extension $L/K_\infty$ in the sense of Ferraro (see \cite[Definition 3.3]{FER1}),
\begin{equation}\label{limit-V-star-m}
\lim_{i\rightarrow\infty}V^*_{m_i}=V^*,
\end{equation}
where $V^*$ is one of the $h_A$ divisors of Lemma \ref{lemma-dual-shtuka}. We do not review the notion of compact topology over $(X_{\CC_\infty}^g/\mathfrak{S}_g)(L)$ but note that, by the fact that $X_{\CC_\infty}^g/\mathfrak{S}_g$ is proper, identifying
$$(X_{\CC_\infty}^g/\mathfrak{S}_g)(L)\cong\varprojlim(X_{\CC_\infty}^g/\mathfrak{S}_g)(\mathcal{O}_L/\mathfrak{m}_L^n),$$
a metric arises over $(X_{\CC_\infty}^g/\mathfrak{S}_g)(L)$, where $\mathcal{O}_L$ and $\mathfrak{m}_L$ are, respectively,
the valuation ring and the maximal ideal associated to $L$.

Now we go back to our partial zeta function
$Z_{X,A}$. Let $P$ be any point in the support of $V^*$ (so we assume that $g>0$). By Lemma \ref{partial-zeta-are-entire} we can extend $Z_{X,A}$ analytically over $X(\CC_\infty)\setminus\{\infty\}$. Let us look at its evaluation at $P$ and $P_i$. Up to selection of a subsequence we can choose, 
for all $i$ a point $P_i$ in the support of $V^*_{m_i}$ with the property that $P_i\rightarrow P$ (compact topology).
For any such $i$:
$$Z_{X,A}(P)=\underbrace{Z_{X,A}(P)-Z_{X,A}(P_i)}_{T_1(i)}+\underbrace{Z_{X,A}(P_i)-\mathcal{S}_{\leq j_{m_i}}(P_{i})}_{T_2(i)}+\underbrace{\mathcal{S}_{\leq j_{m_i}}(P_{i})}_{T_3(i)}.$$ As $i\rightarrow\infty$, $T_1(i)\rightarrow0$ because $P_i\rightarrow P$ and $Z_{X,A}$ is entire, hence continuous. We also have $T_2(i)\rightarrow0$ because for all $\rho\in|\CC_\infty^\times|$,
$\|Z_{X,A}-\mathcal{S}_{\leq j_{m_i}}\|_\rho\rightarrow0$ and for all $i$ big, we can find $\rho$ such that $P_i\in\boldsymbol{D}_X(\rho)$ (this follows easily from the proof of \cite[Lemma 1.1]{CHU&NGO&PEL})
and finally, $T_3(i)=0$ for all $i$ large by the choice of $P_i$. We deduce that $Z_{X,A}(P)$ vanishes. Hence we get the following partial portrait of $Z_{X,A}$:

\begin{Lemma}
The entire function $Z_{X,A}$ can be identified with an element of $\TT_A^\times$ and vanishes at the points of one of the $h_A$ divisors $V^*$, and on all the points 
$\Xi^{(1)},\Xi^{(2)},\ldots$.
\end{Lemma}

Of course, this does not rule out the existence of other zeroes, or higher multiplicities but it turns out that this is the precise description of the zeroes of $Z_{X,A}$, see Corollary \ref{Corollary-precise-divisors}. Note that if $I\neq A$, $Z_{X,I}$ needs not to belong to $\TT_A^\times$. This more general case is discussed in detail in \cite{FER1}.

\subsection{Sketch of proof of Ferraro's Theorem} It is time to outline a proof of Theorem \ref{Ferraro-theorem} (full details can be found in \cite{FER1}, here we only care about the essential steps and a simplified statement). Recall \eqref{divisor-of-Sj}, choose a shtuka function $f$ satisfying \eqref{shtuka-function} and consider $\delta$ as in Lemma \ref{corollary-Weil}. These functions are defined up to the choice of a proportionality factor in $\CC_\infty^\times$ we will be a little bit indeterminate about these choices that can be made at the final step.
For $m$ large, writing $$h_m:=\frac{\mathcal{S}_{\leq j_{m}}\delta^{(m+1)}}{f^{(1)}\cdots f^{(m)}}$$ we have, by a simple computation using \eqref{divisor-partial-sums},
\begin{equation}
(h_m)={V^*}^{(1)}_m+V^{(1)}+{V^*}^{(m+1)}-3g\infty.
\end{equation}
Therefore $h_m\in A(\leq 3g)\otimes_{\FF_q}\CC_\infty$ for all $m$. 

We now use results that Ferraro developed to control convergence of sequences of rational functions under the hypothesis of convergence of the associated sequences of divisors for the compact topology; we do not review them here but we discuss the useful consequences. We can apply his Proposition 3.20 by setting, in his notation,
$D_:=3g\infty$, $d=3g$. For the compact topology, there is a strictly increasing subsequence $m_i$ with the property that
${V^*}^{(1)}_{m_i}+V^{(1)}+{V^*}^{(m_i+1)}\rightarrow {V^*}^{(1)}+V^{(1)}+W$ in $X^{3g}/\mathfrak{S}_{3g}$ where $W$ is the
effective divisor of degree $g$ in $(X^{g}/\mathfrak{S}_{g})(\FF_q)$ such that ${V^*}^{(m_i+1)}\rightarrow W$ (note that
${v^*}^{(m_i)}\rightarrow0$ so $[W-g\infty]=0$). In the notation of \cite[Proposition 3.20]{FER1}, we can set
$D_+={V^*}^{(1)}+V^{(1)}+W$. The first part of this result implies that there exists a sequence of elements $\lambda_{i}\in\CC_\infty^\times$ such that $\kappa_i:=\lambda_ih_{n_i}\rightarrow \kappa$ with $\kappa\in\mathcal{L}(3g\infty)$
with respect of any norm on $\mathcal{L}(3g\infty)$ (they are all equivalent as this is a $\CC_\infty$-vector space of finite dimension); we take the Gauss norm on $\mathcal{L}(3g\infty)=A(\leq 3g)\otimes_{\FF_q}\CC_\infty$ after a choice of a basis of $A(\leq 3g)$ that we extend to a basis of $A$. 

Note that
 $\kappa$ is such that
$$\Div(\kappa)={V^*}^{(1)}+V^{(1)}+W-3g\infty.$$
By Lemma \ref{tau-difference-equation}, we have thus obtained, along an increasing sequence $m_i$, 
$$\lambda_i\omega^{(1)}\mathcal{S}_{\leq j_{m_i}}\delta^{(m_i+1)}=\kappa_i\omega^{(m_i+1)},$$
and $\kappa_i\rightarrow\kappa$ for the chosen Gauss norm over $\mathcal{L}(3g\infty)$. We know that $\mathcal{S}_{\leq j_{m_i}}$ converges to $Z_{X,A}$ for the Gauss norm of $\TT_A$. Also $\omega^{(j)},\delta^{(j)}$ are elements of $\TT_A$ for all $j$. In particular there exists
$\alpha\in L_\infty^\times$ such that $(\alpha\omega)^{(m_i+1)}\rightarrow\overline{\omega}\in A\otimes_{\FF_q}\FF_q^{\text{alg}}$ (upon extraction of a strictly increasing subsequence of $m_i$).
For a similar reason, by the fact that $\delta\in A(\leq 2g)\otimes_{\FF_q}\CC_\infty$ we can find an element $\beta\in\CC_\infty^\times$ such that, upon extraction of a subsequence, $(\beta\delta)^{(m_i+1)}\rightarrow\overline{\delta}\in A\otimes_{\FF_q}\FF_q^{\text{alg}}$. We deduce that, along a subsequence, $\lambda_i\big(\frac{\alpha}{\beta}\big)^{(m_i+1)}$
tends to an element $\lambda\in\CC_\infty^\times$, and we reach the equality in $\TT_A$, which can be seen already as a functional identity (the right-hand side is a rational function):
$$
\omega^{(1)}Z_{X,A}=\lambda^{-1}\frac{\overline{\omega}}{\overline{\delta}}\kappa.
$$
We come back to the divisor of $\kappa$. Recall that $[W-g\infty]=0=\lim_i{V^*}^{(m_i+1)}$. Hence there exists a rational function $\gamma$ in $A\otimes_{\FF_q}\CC_\infty$
such that $(\gamma)=W-g\infty$. But we can choose $\gamma\in A\otimes_{\FF_q}\FF_q^{\text{alg}}$. So we find that
$$\kappa=\gamma\hat{\kappa}$$
where $\hat{\kappa}$ is a rational function such that $\Div(\hat{\kappa})={V^*}^{(1)}+V^{(1)}-2g\infty$. By Lemma \ref{corollary-Weil},
$\hat{\kappa}$ and $\delta^{(1)}$ are proportional.
With this we can write:
\begin{equation}\label{partial-functional-identity}
\omega^{(1)}Z_{X,A}=\mu^{-1}\overline{\eta}\delta^{(1)},
\end{equation}
for $\mu\in\CC_\infty^\times$ and $\overline{\eta}\in\operatorname{Frac}(A\otimes_{\FF_q}\FF_q^{\text{alg}})$.
Now recall that $Z_{X,A}\in\TT_A^\times$. Hence, it does not vanish in $\boldsymbol{D}_X$. Since $X(\FF_q^{\text{alg}})\setminus\{\infty\}\subset\boldsymbol{D}_X$, it does not vanish on $X(\FF_q^{\text{alg}})\setminus\{\infty\}$ neither.
By \cite[Proposition 4.23]{FER1}, the support of $V$ does not meet $X(\FF_q^{\text{alg}})\setminus\{\infty\}$.
Hence the set of zeroes and poles of $\omega$ in $X(\FF_q^{\text{alg}})\setminus\{\infty\}$ determines a divisor $D$ defined over $\FF_q$. By using \cite[Proposition 4.16]{FER1} we can find $s\in\ZZ$ such that $D+s\infty=\Div(\nu)$ with $\nu\in \operatorname{Frac}(A\otimes1)$ rational over $X$. Since $\nu^{(1)}=\nu$, we can choose $\omega$ without zeroes on 
$X(\FF_q^{\text{alg}})\setminus\{\infty\}$ and with this choice, in the formula \eqref{partial-functional-identity} we see that
$\overline{\eta}$ can be replaced with $1$ upon choosing appropriate normalizations. 

With a little additional work the reader can verify, using the formula \eqref{functional-identity-Ferraro}, the following complete description: 
\begin{Corollary}\label{Corollary-precise-divisors}
The function $Z_{X,A}$, away from $\infty$, has the infinite support divisor $${V^*}^{(1)}+\sum_{i>0}\Xi^{(i)}$$ and there exists a special function $\omega\in\mathfrak{Sf}(\phi)$ that has, away from $\infty$, infinite support divisor $${V}-\sum_{i\geq 0}\Xi^{(i)}.$$
\end{Corollary}

We end here our overview on Ferraro's proof of his Theorem \ref{Ferraro-theorem}. We omit recognizing that the 
right hand side of \eqref{partial-functional-identity} belongs to $\widetilde{\pi}(A(\leq 2g)\otimes_{\FF_q} H)$. But we gave an idea of how the functional identity he proved is connected with the geometric structure of the curve $X$. As a byproduct, 
Theorem \ref{Ferraro-theorem} delivers fragments of Goss' Theorem \ref{Goss-formula-Euler}. It contains a special case of 
\eqref{partial-result-of-goss} (\footnote{The theory presented in this paper can fully cover the above theorem of Goss but the price to pay is to introduce several variables generalizations of the functions we are using, see \cite{ANG&PEL} and \cite{ANG&NGO&TAV1}.}):

\begin{Corollary}
For all $k\geq1$ we have 
$$Z_{A}(q^k-1)\in H\widetilde{\pi}^{q^k-1}.$$
\end{Corollary}

\begin{proof}
The function $f\omega$ has no pole and no zero at $\Xi$.
From \cite[Theorem 5.8]{GAZ&MAU2} we deduce that, choosing $f\in\operatorname{Frac}(A\otimes H)$, 
$$(f\omega)(\Xi)\in H^\times\widetilde{\pi}.$$ Now apply the $k$-th power $\tau^k$ of Anderson's twist $\tau$ to both sides 
of \eqref{functional-identity-Ferraro}. We get:
$$Z^{(k)}_{X,A}=\widetilde{\pi}^{q^k}\frac{\delta^{(1+k)}}{f\cdots f^{(k)}\omega}.$$
We conclude evaluating at $\Xi$, because $Z^{(k)}_{X,A}(\Xi)=Z_A(q^k-1)$.
\end{proof}

In fact, the arguments can be reversed: assuming Goss' Theorem \ref{Goss-formula-Euler} we deduce that 
$$\frac{Z_{A}(q^k-1)}{\widetilde{\pi}^{q^k-1}}\in H$$ for all $k>0$. With good normalizations we deduce that 
$\delta^{(k)}(\Xi)\in H$ for all $k>0$. This implies that we can find $\delta\in A\otimes_{\FF_q} H$.

\section{Dedekind-like Zeta functions over curves}\label{Dedekind}

The construction of the functions $Z_{X,I}$ discussed in the previous section follow Goss' path, when he constructed the functions $Z_I$ in \S \ref{Goss-partial-zeta}.
We do not know if we can construct, analogously, a zeta function of the type of Goss' $\zeta_A$ in \eqref{definition-Goss-zeta}, over $X(\CC_\infty)\setminus\{\infty\}$. But we can at least construct relative partial zeta functions and a Dedekind-like such a function $\zeta_{E,A}$ corresponding to a finite separable field extension $E/K$, provided that it contains the Hilbert class field $H$ of $K$. Note that these are the settings of the paper of Angl\`es, Ngo Dac and Tavares Ribeiro \cite{ANG&NGO&TAV}. These functions first appeared in ibid. 

We suppose that a triple $(X,\infty,\sgn)$ is given. 
We recall that the Hilbert class field $H$ of $A$ is the maximal abelian extension of $K$ unramified at all the primes of $A$ and is totally split at $\infty$, which means that, identifying $H$ with a subfield of $\CC_\infty$, we have $H\subset K_\infty$, and moreover, $H\otimes_{K}K_\infty$ is a $K_\infty$-vector space of dimension $[H:K]$ which is the order $h_A$ of the class group of $K$. We know that Artin's symbol provides 
a group homomorphism $\sigma:\mathcal{I}_A\rightarrow G:=\operatorname{Gal}(H/K)$. If $I$ is an invertible ideal of $A$, we write $\sigma_I$ for its image in $G$. 
 
Let $M$ be a finite $A$-module. It is isomorphic to a finite direct sum $\oplus_i A/I_i$ for ideals $I_i$ of $A$. We write $$[M]_A=\prod_iI_i\in\mathcal{I}_A.$$ This defines an exact function from finite $A$-modules to $\mathcal{I}_A$.
We return to our field $E$ which is a finite separable extension of $H$, and denote by $B$ the integral closure of $A$ in $E$.
For $\mathfrak{I}$ an ideal of $B$, $B/\mathfrak{I}$ is a finite $A$-module so that $[B/\mathfrak{I}]_A$ is well defined as an ideal of $A$.
It is a principal ideal of $A$. Indeed, factorizing in $\mathcal{I}_B$, it suffices to see that if $\mathfrak{P}$
is a maximal ideal of $B$, then $[B/\mathfrak{P}]_A$ is principal. This is because it is equal to $$N_{E/K}(\mathfrak{P})=\mathfrak{p}^{[B/\mathfrak{P}:A/\mathfrak{p}]}$$ where $\mathfrak{p}$ is the maximal ideal of $A$ such that $\mathfrak{P}\mid\mathfrak{p}$. Now,  
$\mathfrak{p}^{[B/\mathfrak{P}:A/\mathfrak{p}]}$ is a principal ideal 
by a classical property of Hilbert class fields, as $E$ is an extension of $H$ by hypothesis, see \cite{ANG&NGO&TAV}. So we denote,
for $\mathfrak{I}$ a non-zero ideal of $B$,
$[B/\mathfrak{I}]_A^+\in A^+$ to be the unique monic generator of $[B/\mathfrak{I}]_A$. This extends to a group homomorphism $\mathcal{I}_B\rightarrow K^\times\cap\operatorname{Ker}(\sgn)$.
Define, following \cite[\S 4.2]{ANG&NGO&TAV},
$$\zeta_{E,A}(\chi):=\sum_{\begin{smallmatrix} \mathfrak{I} \text{ ideal of } B\end{smallmatrix}}\frac{[B/\mathfrak{I}]_A^+\otimes1}{1\otimes [B/\mathfrak{I}]_A^+}=\prod_{\begin{smallmatrix} \mathfrak{P} \text{ maximal}\\\text{ideal of } B\end{smallmatrix}}\Big(1-\frac{[B/\mathfrak{P}]_A^+\otimes1}{1\otimes [B/\mathfrak{P}]_A^+}\Big)^{-1}.$$
The factors of the product are elements of $1+A\otimes_{\FF_q}\mathfrak{m}_{K_\infty}$ 
(with $\mathfrak{m}_{K_\infty}$ the maximal ideal of $K_\infty$) and the product converges to a unit of the Tate algebra $\TT_A$. Indeed for the Gauss valuation $\|\cdot\|$ of $\TT_A$,
$$\left\|\frac{[B/\mathfrak{P}]_A^+\otimes1}{1\otimes [B/\mathfrak{P}]_A^+}\right\|<1\quad \text{ and }
\left\|\frac{[B/\mathfrak{P}]_A^+\otimes1}{1\otimes [B/\mathfrak{P}]_A^+}\right\|\rightarrow0$$ for 
$\mathfrak{P}$ maximal (the limit runs over an arbitrary choice of sequence of distinct maximal ideals). Each factor, seen as a function over the domain $D\subset X(\CC_\infty)\setminus\{\infty\}$, is nowhere vanishing.
As functions, product and series 
converge for $\chi\in D$. 

We show the following result, which is likely to be also reacheable with the methods introduced in \cite[\S 4.2]{ANG&NGO&TAV}:

\begin{Theorem}\label{zeta-E-A-are-entire}
The function $\zeta_{E,A}$ extends to an entire function $X(\CC_\infty)\setminus\{\infty\}\rightarrow\CC_\infty$.
\end{Theorem}

%This result implies that the functions $\zeta_{E,A}$ are entire.

\begin{proof}%[Proof of Theorem \ref{zeta-E-A-are-entire}]
In virtue of the forthcoming decomposition \eqref{dec-total-partial-2} and Theorem \ref{lemma-relative-partial-is-entire}, $\zeta_{E,A}$ defines a series which is everywhere convergent on $X(\CC_\infty)\setminus\{\infty\}$ outside a finite set of points of
$X(\FF_q^{\text{alg}})$ where it may have poles. However its Euler product expansion defines an element of the Tate algebra $\TT_A$. Hence the above finite set of possible poles is empty.
\end{proof}

\begin{Remark}{\em 
We still need to prove Theorem \ref{lemma-relative-partial-is-entire}. Our proof of
Theorem \ref{lemma-relative-partial-is-entire} is obtained adapting the methods of \cite[\S 8.9]{GOS}.
In \cite[\S 4.2]{ANG&NGO&TAV}, the authors originally presented an alternative approach using several variables log-algebraicity theorem even introducing at once a multi-variable generalization of the functions $\zeta_{E,A}$. 

In the case $E=H$ the authors of \cite{ANG&NGO&TAV} introduce a generalization of our functions which can be identified with an element of the following ultrametric Banach algebra. Consider the ring $$\mathbb{R}_{A}:=\widehat{\FF_q[z]\otimes_{\FF_q}\operatorname{Frac}(A\otimes_{\FF_q}K)},$$ where $z$ is an indeterminate and where we take the completion with respect to the Gauss valuation that extends the $\infty$-adic valuation we choose on $K$, identified with $1\otimes1\otimes K$. They introduce indeed an element $\tilde{\zeta}_{H,A}$ of $\mathbb{R}_{A}$ for which the value at $z=1$ is meaningful and equals $\zeta_{H,A}$ seen as an element of 
$\widehat{\operatorname{Frac}(A\otimes_{\FF_q}K)}$. The interest of this element (see \cite[Proposition 4.7]{ANG&NGO&TAV}) is that it is the determinant of an {\em equivariant Anderson harmonic series}, seen as an endomorphism of a certain complete group algebra over $G$. If it is possible to identify the equivariant Anderson harmonic series with meromorphic functions, this result of 
\cite{ANG&NGO&TAV} covers our result but we have not tried to pursue in this direction. The particular consideration of the variable $z$ in the above opens, as these authors prove in a large spectrum of papers culminating with \cite{ANG&NGO&TAV3} (see the references therein) which is, up to the author knowledge, the most general version known of a class number formula in the sense of Taelman \cite{TAE}. Ashamedly, we refrain from giving more details here,
to limit the size of our presentation.
}\end{Remark} 

\subsection{Relative partial zeta functions}
We show here that these functions too, extend to entire functions. We use the same notations as before. Consider $\mathfrak{I}$ an ideal of $B$. We define
$$Z_{\mathfrak{I},A}:=\sum_{(b)\subset \mathfrak{I}}\frac{[B/(b)]_A^+\otimes1}{1\otimes [B/(b)]_A^+}\in\TT_A,$$
where the sum runs over the non-zero principal ideals of $B$ contained in $\mathfrak{I}$. 
These are variants of the Goss partial zeta functions discussed in \S \ref{Goss-partial-zeta}. Let 
$\mathcal{I}_1,\ldots,\mathcal{I}_s$ be cosets of the ideal group $\mathcal{I}_B$ modulo principal invertible ideals (there are finitely many). Choose
representatives $\mathfrak{I}_1,\ldots,\mathfrak{I}_s$ that are invertible ideals with the property that their inverses are ideals of $B$.
Then, just as in \eqref{dec-total-partial},
\begin{equation}\label{dec-total-partial-2}
\zeta_{E,A}:=\sum_{i=1}^s\frac{[B/\mathfrak{I}_i]_A^+\otimes1}{1\otimes [B/\mathfrak{I}_i]_A^+}Z_{\mathfrak{I}_i^{-1},A}
\end{equation}
in $\TT_A$.
Our task in this subsection is to show the following. It is not clear that the methods of \cite{ANG&NGO&TAV} can be used to prove it.

\begin{Theorem}\label{lemma-relative-partial-is-entire}
Let $E/H/K$ be a tower of finite extensions with $E/K$ separable, let $B$ the integral closure of $A$ in $E$.
 For $\mathfrak{I}$ ideal of $B$, $Z_{\mathfrak{I},A}$ extends to an entire function over $X(\CC_\infty)\setminus\{\infty\}$.
\end{Theorem}

The proof of the theorem requires additional tools, mainly, we need to revisit the tools of
\cite[Lemma 8.9.3]{GOS}.

\subsubsection*{Subvector spaces} We consider the diagonal embedding $\Sigma:E\rightarrow E_\infty:=E\otimes_{K}K_\infty$.
We can identify $E_\infty:=\oplus_{i=1}^h K_{i,\infty}$ with $K_{i,\infty}$ a finite extension of $K_\infty$ with ramification index $e_i$ and residual degree $f_i$, and write $\Sigma=(\sigma_i)_i$. For example, if $E=H$, $h=h_A$, $f_i=e_i=1$ for all $i$ (the infinity place of $K$ totally splits in $H$).
Choosing for all $i$ a uniformizer $\pi_i$ of $K_{i,\infty}$ with the associated sign function $\sgn_i:K_{i,\infty}^\times\rightarrow k_i^\times$ (with $k_i$ the residual field of $K_{i,\infty}$), we can associate to 
every element $x=(x_i)\in E_\infty$ the vector $$\partial(x):=(\deg_i(x_i))_i\in(\ZZ\cup\{-\infty\})^h$$ where, if
$x_i$ is non-zero, $\deg_i(x_i)$ is the unique integer such that $x_i\pi_i^{\deg_i(x_i)}$ has multiplicative valuation $1$ in $K_{i,\infty}$ (degree $0$), while the value $-\infty$ is associated to $0$. If $x$ is an element of $E$ we identify it with 
its image $\Sigma(x)$ in $E_\infty$.
For $x,y\in E_\infty$ we write $x<y$ if coefficientwise, $\partial(x)<\partial(y)$. To fix the ideas, if $E=K=\FF_q(\theta)$ so
$A=\FF_q[\theta]$, then $a<b$ if and only if $\deg(a)<\deg(b)$.

We set $Y:=\operatorname{Spec}(B)$. Then there is a finite covering 
$Y\rightarrow X$. The embeddings $\sigma_i$ are in bijective correspondence with the points $\infty_i$ of $Y$ above $\infty$. For example if $E=H$ we have $h_A$ distinct points $\infty_i$ above $\infty$. In all cases the datum of $\partial(x)$ corresponds to the infinity part of the divisor of $x$, seen as a rational function over $Y$.
For two rational functions $x,y$ over $Y$, $x<y$ is equivalent to the property that, at each point $\infty_i$ above $\infty$, the order of $x$ is strictly larger than that of $y$.
Note that the image of any invertible ideal $\mathfrak{I}$ in $E_\infty$ via $\Sigma$ is a discrete and co-compact submodule
of $E_\infty$.

The map $\partial: E\rightarrow(\ZZ\cup\{-\infty\})^h$ defined in this way behaves like the opposite of an additive valuation. Consider $x,y\in E$. Then (1) $\partial(x)$ has a vanishing component if and only if $x=0$. (2) $\partial(xy)=\partial(x)+\partial(y)$. (3) If $\partial(x)<\partial(y)$ then $\partial(y)=\partial(x+y)$.

We keep considering $\mathfrak{I}$ a non-zero ideal of $B$. We write
$$\Lambda(b):=\{b'\in\mathfrak{I}:\partial(b')=\partial(b)\text{ and }\sgn_i(b')=\sgn_i(b)\text{ for all $i$}\}.$$
Note that for all $b',b''\in\Lambda(b)$, $\partial(b'-b'')<\partial(b)$ and for all $b'\in \Lambda(b)$, $\Lambda(b')=\Lambda(b)$. Hence, $$\Lambda(<b):=\{b'-b'':b',b''\in\Lambda(b)\},$$ is a finite dimensional vector space over $\FF_q$ and for all $b\in \Lambda(b)$, 
$$\Lambda(b)=b+\Lambda(<b).$$

\subsubsection*{Principal ideals}
For a non-empty subset $\mathfrak{S}$ of $\mathfrak{I}\setminus\{0\}$ we denote by $\Pi(\mathfrak{S})$ the set of principal ideals 
generated by the elements of $\mathfrak{S}$. They are all contained in $\mathfrak{I}$. If $d\in\NN$ we denote by $\Pi(\mathfrak{S})(d)$ the set of principal ideals $(b)$ of $\Pi(\mathfrak{S})$ such that their norm $N_{E/K}(b)\subset A$ has degree $d$ (this set can be empty).

\begin{Lemma}\label{lemma-goss-1}
Consider $(b)\subset\mathfrak{I}$ a non-zero principal ideal in $\mathfrak{I}$. Write $d=\deg(N_{E/K}(b))\in\NN$.
The map 
$$\Lambda(b)\xrightarrow{\mu}\Pi(\mathfrak{I}\setminus\{0\})(d)$$
defined by sending an element $b'$ to the corresponding principal ideal $(b')\subset\mathfrak{I}$ is injective.
\end{Lemma}

\begin{proof}
Choose $b_0\in\mathfrak{I}\setminus\{0\}$ with $d=\deg(N_{E/K}(b_0))$
and suppose that $b,b'\in\Lambda(<b_0)$ are such that we have an equality of principal ideals $(b_0+b)=(b_0+b')$. Then there exists $\eta\in B^\times$ such that $b_0+b=\eta(b_0+b')$. We get $\partial(b_0)=\partial(b_0+b)=\partial(\eta(b_0+b'))=\partial(\eta)+\partial(b_0+b')=
\partial(\eta)+\partial(b_0)$. This means that $\partial(\eta)=0$, that is, $\eta\in\FF_q^\times$. Coming back to the identity
$b_0+b=\eta(b_0+b')$ now that we learned that $\eta\in\FF_q^\times$, we have two possibilities: $\eta=1$. In this case we obtain $b=b'$ and the claimed injectivity. If $\eta\neq1$ then $(1-\eta)b_0=\eta b'-b$ but $\partial(b_0)=\partial((1-\eta)b_0)=\partial(\eta b'-b)<\partial(b_0)$, which is impossible. 
\end{proof}

\begin{Lemma}\label{lemma-goss-2}
Consider two non-zero elements $b_0,b_0'\in\mathfrak{I}$. We have that $\Pi(\Lambda(b_0)\setminus\{0\})\cap\Pi(\Lambda(b_0')\setminus\{0\})\neq\emptyset$ if and only if
$\Pi(\Lambda(b_0)\setminus\{0\})=\Pi(\Lambda(b_0')\setminus\{0\})$.
\end{Lemma}

\begin{proof}
Suppose that we have, for $b\in\Lambda(<b_0)$ and $b'\in\Lambda(<b_0')$, and identity of principal ideals
$(b_0+b)=(b_0'+b')$. 
Then there exists $\eta\in B^\times$ such that $\partial(b_0)=\partial(b_0+b)=\partial(\eta(b_0+b))$.
But $\partial(\eta(b_0+b))=\partial(\eta)+\partial(b_0+b)=\partial(\eta)+\partial(b_0)=\partial(\eta b_0)$.
Hence $\partial(\eta b')<\partial(\eta b'_0)=\partial(b_0)$. Since we also have that $\partial(b)<\partial(b_0)$,
we have $\partial(b'')<\partial(b_0)$ where $b''=b-\eta b'$. We deduce that $\eta b_0'=b_0+b''\in\Lambda(b_0)$
so that $\eta b_0'\in\Lambda(b_0)\setminus\{0\}$. For all $b'$ such that $b'<b_0'$ we have $\eta b'<b_0$ so that
$\eta (b_0'+b')\in\Lambda(b_0)\setminus\{0\}$. This means that $\eta(\Lambda(b_0')\setminus\{0\})\subset\Lambda(b_0)\setminus\{0\}$ and therefore $\Pi(\Lambda(b_0')\setminus\{0\})\subset\Pi(\Lambda(b_0)\setminus\{0\})$.
But exchanging the roles of $b_0,b_0'$, we obtain equality.
The reverse implication is clear.
\end{proof}

We just learned that over $\Pi(\mathfrak{I}\setminus\{0\})(d)$, we have an equivalence relation $\approx$ defined by $(b)\approx (b')$ if and only if $\Pi(\Lambda(b)\setminus\{0\})\cap\Pi(\Lambda(b')\setminus\{0\})\neq\emptyset$, with finitely many classes that can be labelized with representatives
$(b_{d,1}),\ldots,(b_{d,l_d})$. Combining Lemmas \ref{lemma-goss-1} and \ref{lemma-goss-2} we obtain the decomposition in disjoint subsets:
\begin{eqnarray}
\Pi(\mathfrak{I}\setminus\{0\})&=&\bigsqcup_{d\geq0}\Pi(\mathfrak{I}\setminus\{0\})(d)\nonumber\\
&=&\bigsqcup_{d\geq0}\bigsqcup_{i=1}^{l_d}\Pi(\Lambda(b_{d,i})\setminus\{0\})\nonumber\\
&=&\bigsqcup_{d\geq0}\bigsqcup_{i=1}^{l_d}\{(b):b\in\Lambda(b_{d,i})\setminus\{0\}\}\label{decomposition-bdi}.
\end{eqnarray}
We deduce that there is a decomposition:
\begin{equation}\label{goss-decomposition}
Z_{\mathfrak{I},A}=\sum_{d\geq 0}\sum_{i=1}^{l_d}S_{\Lambda(b_{d,i})}(1),
\end{equation}
where for $\Lambda:=\Lambda(b_0)$, $b_0\in B$ and $n>0$, 
$S_{\Lambda}(n)$ is the finite sum
$$S_{\Lambda}(n):=\sum_{b\in\Lambda\setminus\{0\}}\frac{[B/(b)]^+_A\otimes1}{1\otimes\Big([B/(b)]^+_A\Big)^n}\in A\otimes_{\FF_q} K.$$

\subsubsection*{Growth estimates}
Given a sequence $(i_d)_d$ with $1\leq i_d\leq l_d$, crucial is now the understanding of the growth of
$\|S_{\Lambda(b_{d,i})}(n)\|_\rho$ for $d\rightarrow\infty$,
where $\rho\in|\CC_\infty^\times|$, $\|\cdot\|_\rho$ is the multiplicative valuation of $A\otimes_{\FF_q}\CC_\infty$ extending that chosen for $1\otimes\CC_\infty$ and determined by $\|a\otimes1\|_\rho=\rho^{\deg(a)}$ for $a\in A$ (we are going to follow the path of \cite[Proof of Lemma 1.1]{CHU&NGO&PEL}). In fact we need an extension of the functions $S_{\Lambda(b_{d,i})}:\ZZ\rightarrow K_\infty$ to
Goss space of quasi-characters $\mathbb{S}_\infty$. We set, for $s=(x,y)\in\mathbb{S}_\infty$:
$$S_{\Lambda}(s):=x^{-d}\sum_{b\in\Lambda}\frac{[B/(b)]_A^+\otimes1}{1\otimes\langle [B/(b)]_A^+\rangle^y}\in A\otimes\CC_\infty,$$
where $d=\deg([B/(b)]_A^+)$. Note that if $s=(\pi^{-n},n)$ as in (\ref{definition-sn}) then
$S_{\Lambda}(s_n)$ agrees with $S_{\Lambda}(n)$. For the sake of simplicity, we now focus on the case $n=1$.

\begin{Proposition}\label{rho-r-E}
Choose $\rho\in|\CC_\infty^\times|$, $\rho\geq1$, and $r\geq 1$. There exists a constant $c_1=c_1(E)>0$ depending on $E$ only, such that for all $d\geq c_1r$ and for all $i\in\{1,\ldots,l_d\}$ we have
$$\|S_{\Lambda(b_{d,i})}(1)\|_\rho\leq (\rho|\pi|)^d|\pi|^{q^{r+1}},$$
where the $(b_{d,i})$ are the representatives that yield the decomposition (\ref{decomposition-bdi}). 
\end{Proposition}

\begin{proof}
Write $\Lambda=\Lambda(b_{d,i})$ and consider $b\in\Lambda$. Then, with $a_b:=[B/(b)]_A^+\in A^+$, $d=\deg(a_b)$,
$$(1\otimes\pi^{d})\frac{a_b\otimes1}{1\otimes\langle a_b\rangle}=\frac{a_b\otimes1}{1\otimes a_b}=\frac{[B/(b)]_A^+\otimes1}{1\otimes [B/(b)]_A^+}=\frac{N_{E/K}(b)\otimes1}{1\otimes N_{E/K}(b)}=\frac{\prod_i\sigma_i(b)\otimes1}{1\otimes\prod_i\sigma_i(b)}.$$
Now compute, with $s_1=(\pi^{-1},1)\in\mathbb{S}_\infty$,
$$S_\Lambda(s_1)-S_\Lambda((\pi^{-1},1-q^{r+1}))=(1\otimes\pi^{d})\sum_{b\in\Lambda}\frac{a_b\otimes1}{1\otimes\langle a_b\rangle}\Big(1\otimes(1-\langle a_b\rangle)^{q^{r+1}}\Big).$$
Moreover
\begin{eqnarray*}
\pi^{-dq^{r+1}}S_\Lambda((\pi^{-1},1-q^{r+1}))&=&\pi^{d(1-q^{r+1})}\sum_{b\in\Lambda}(a_b\otimes1)1\otimes\langle a_b\rangle^{q^{r+1}-1}\\
&=&\sum_{b\in\Lambda}\prod_jL_{0,j}(b)\prod_j\prod_{k=0}^rL_{1,j,k}(b)^{q-1},
\end{eqnarray*}
where $L_{0,j},L_{1,j,k}$ are the $\FF_q$-linear forms
$E\rightarrow E_\infty$ defined by $L_{0,j}(b)=\sigma_j(b)\otimes1$ with $\sigma_j$ the $j$-th $K$-embedding
and $L_{1,j,k}(b)=1\otimes\sigma_j(b)^{q^k}$. Therefore to control the vanishing of $S_\Lambda((\pi^{-1},1-q^{r+1}))$
we can use \cite[Lemma 8.8.1]{GOS} with $J=E,J_1=E_\infty$ (it is not a field, but the given proof  extends to our setting), $\{\mathcal{L}_1,\ldots,\mathcal{L}_t\}=\{L_{0,j},L_{1,j,k}:\text{ for varying }j,k\}$
so that $t=[E:K](1+(r+1)(q-1))$ and if 
$$\dim_{\FF_q}(\Lambda(<b_{d,i}))>[E:K](1+(r+1)(q-1))$$
then $S_\Lambda((\pi^{-1},1-q^{r+1}))$ vanishes. We now discuss when this happens. With the previous choice of $d$, we recall that all elements of $\Lambda(<b_{d,i})$ can be identified with rational functions on $Y$ in the $\FF_q$-vector space
$\mathcal{L}(\Delta(b_{d,i}))$ composed of those rational functions $f$ with $(f)+\Delta(b_{d,i})$ effective ($(f)$ denotes here the divisor associated to $f$),
where $D(b_{d,i})$ is the polar divisor of $b_{d,i}$. In fact, the image of $\Lambda(<b_{d,i})$ in $\mathcal{L}(D(b_{d,i}))$ has finite co-dimension uniformly bounded in terms of the number of points above $\infty$. By Riemann-Roch Theorem, $\dim_{\FF_q}(\mathcal{L}(D(b_{d,i})))\geq\deg(D(b_{d,i}))-g'+1=d-g'+1$ independent on $i$. Therefore there is a constant $c_1$ depending on $E$ such that if $d\geq c_1r$, $S_\Lambda((\pi^{-1},1-q^{r+1}))$ vanishes for $\Lambda=\Lambda(b_{d,i})$ independently on $i$. 

Suppose that $d\geq c_1r$. 
Then
\begin{eqnarray*}
S_{\Lambda}(1)=S_{\Lambda}((\pi^{-1},1))&=&S_{\Lambda}((\pi^{-1},1))-S_{\Lambda}((\pi^{-1},1-q^{r+1}))\\
&=&(1\otimes\pi^d)\sum_{b\in\Lambda}\frac{a_b\otimes1}{1\otimes\langle a_b\rangle}\Big(1\otimes(1-\langle a_b\rangle)^{q^{r+1}}\Big).
\end{eqnarray*}
Choose $\rho\geq1$. We have, for $b,a_b$ as above, $\|a_b\otimes1\|_\rho=\rho^d$. Also, $\|1\otimes a_b\|\leq|\pi|^{q^{r+1}}$.
Hence we get, for all $d\geq c_1r$,
$\|S_{\Lambda}(1)\|_\rho\leq(\rho|\pi|)^d|\pi|^{q^{r+1}}$ as expected.
\end{proof}

\begin{proof}[Proof of Theorem \ref{lemma-relative-partial-is-entire}]
Consider $\rho\in|\CC_\infty^\times|$ such that $\rho|\pi|>1$ and choose $\epsilon>0$. 
The set of integers $d$ such that $c_1r\leq d$ and $(\rho|\pi|)^d|\pi|^{q^{r+1}}<\epsilon$ is a (possibly empty) interval $I_r$
of $\NN$:
$$I_r=\left[c_1r,\frac{\log\epsilon}{c_2}+\frac{c_3}{c_2}q^{r+1}\right]\cap\NN,$$ where $\log$ is the usual logarithm $\RR_{>0}\rightarrow\RR$ (so $\log\epsilon$ may be negative).
where $c_2=c_2(\rho)=\log(\rho|\pi|)$ (a positive constant depending only on $\rho$) and $c_3=-\log(|\pi|)>0$.
There exists a positive constant $c_4=c_4(\rho,\epsilon)$ such that for $r\geq c_4$, the intervals $I_r$ cover a half-line inside $\NN$, therefore missing, at most, finitely many integers. Hence if $d\geq c_4$, we have $\|S_{\Lambda(b_{d,i})}(1)\|_\rho<\epsilon$. In particular, for $\rho$ fixed, the sequence 
$\|S_{\Lambda(b_{d,i})}(1)\|_\rho$ tends to zero for any infinite sequence of distinct values of $(d,i)$. 
Since this holds for any choice of $\rho$ large enough, this concludes the proof that $Z_{\mathfrak{I},A}$ is an entire function.
\end{proof}

\begin{Remark}
{\em The proof of Theorem \ref{lemma-relative-partial-is-entire} we gave easily extends to prove that, for all $n\geq 1$, 
$$Z_{\mathfrak{I},A}(n):=\sum_{(b)\subset \mathfrak{I}}\frac{[B/(b)]_A^+\otimes1}{1\otimes ([B/(b)]_A^+)^n}\in\TT_A,$$ so that $Z_{\mathfrak{I},A}(1)=Z_{\mathfrak{I},A}$, extends to an entire function over $X(\CC_\infty)\setminus\{\infty\}$. It would be nice to  find and describe functional identities for these functions, in the spirit of Ferraro's Theorem \ref{Ferraro-theorem}.}
\end{Remark}

To conclude we observe that, up to a certain extent, we can compare Ferraro's rationality theorem for $\omega^{(1)}\zeta_{X,I}$ with Weil's rationality conjecture. More work is required to develop appropriate correspondences for the other two conjectures, the functional equation and Riemann's hypothesis. 

\section{Addendum: Non-trivial zeros}

In \cite[\S 8.24]{GOS} Goss discusses an avatar of Riemann's hypothesis for his function $\zeta_A$ defined in \eqref{definition-Goss-zeta} in the case $A=\FF_q[\theta]$. For any $y\in\ZZ_p$ the zeros $x$ of the function $x\mapsto \zeta_A(x,-y)$
are in $K_\infty$ and simple. On bottom of p. 338 of his book he even formulates a more general conjectural statement, based on intensive computations, for $A$ with class number one. Developing on partial results of Diaz-Vargas \cite{DIA} and Wan \cite{WAN}, Sheats finally demonstrated this conjecture in \cite{SHE}. We refrain from entering this topic here but it is likely that 
non-trivial zeros of our zeta functions distribute over $X_{\CC_\infty}$ following some patterns that may connect to an analogue of Riemann Hypothesis. Statements such as Theorem \ref{Ferraro-theorem} can be used to define what we can qualify as non-trivial zeroes: the zeroes of the function $\delta^{(1)}$ in this case since it corresponds, in our theory, to a multiple of Riemann's $\xi$ (for functions such as $\zeta_{E,A}$ we do not have a complete picture yet, but we may have a similar correspondence and definition). The rigid analytic curve $X_{\CC_\infty}\setminus\{\infty\}$ retracts (in the framework of Berkovich's theory) to a metric graph homeomorphic to a half-line. It makes sense to investigate the position of the retractions of the non-trivial zeroes on this half line. This is much like analyzing something slightly more general than the slopes of a Newton polygon,
therefore it may lead to a question similar to that of Goss. It will be explored in further investigations.

\section{Appendix by G. H. Ferraro: computational study of $\zeta_{H_R,A}$}\label{Appendix}

Let $R\subseteq A$ be an order with conductor $\mathfrak{C}\subseteq A$. We adopt the notation of \S \ref{Dedekind} choosing $E=H_R/K$ the corresponding ring class field (which is defined in the next section, see also \cite{Hayes}), so that $B$ is now the integral closure of $A$ in $H_R$; the field $H_R$ is the analogue of a real cyclotomic field in characteristic $0$. In this appendix we use ideas from \cite{ANG&NGO&TAV} and techniques from \cite{FER1} to prove in an alternative way that $\zeta_{H_R,A}$ is analytic on $X_{\CC_\infty}\setminus\infty$, and to study its zeros. Moreover, we carry out some computations to analyze their behavior.

Goss's version of this zeta function, which we defined in \S 4 as an analytic function from $\mathbb{S}_\infty$ to ${\CC_\infty}$, has trivial zeros at the points $-(q-1)a$ for $a\in\mathbb{Z}_{>0}$, and each of these zeros has multiplicity at least $h_A$ (\cite[Ex. 8.13.7]{GOS}). We prove an analogous statement for the version of $\zeta_{H,A}$ defined over the curve $X_{\CC_\infty}$.

\section{Theory}

We will denote by $\mathcal{I}_R$ the set of invertible ideals of $R$, and $\mathcal{P}_R\subseteq\mathcal{I}_R$ the subset of principal ideals. By slight abuse of notation, we can identify $\mathcal{I}_R$ with the set of ideals of $A$ coprime with $\mathfrak{C}$ and $\mathcal{P}_R$ with the subset of principal ideals $I$ such that $I\cap R$ is a principal $R$-ideal (the extension and contraction induced by the inclusion $R\subseteq A$ make the two pairs of sets canonically isomorphic). 
Given a Drinfeld $R$-module $\phi$ over ${\CC_\infty}$ and $I\in \mathcal{I}_R$, let's denote by $[I]$ its ideal class in $\mathrm{Cl}(R):=\mathcal{I}_R/\mathcal{P}_R$, by $\phi_I\in {\CC_\infty}[\tau]$ the unique monic generator of the left ideal of ${\CC_\infty}[\tau]$ generated by $\{\phi_a\}_{a\in I}$, and by $\partial\phi_I$ its constant coefficient.

Since $\ker(\phi_I)\subseteq\ker(\phi_{aI})$ for all $a\in A$, there is a unique $\phi_a^{[I]}\in {\CC_\infty}[\tau]$ such that $\phi_a^{[I]}\phi_I=\phi_I\phi_a$. It's immediate to check that $\phi^{[I]}$ is itself a Drinfeld module.

A \emph{field of definition} of $\phi$ is a field $L$ such that, for some $c\in\mathbb{C}_\infty^\times$, the coefficients of $c^{-1}\phi_a c$ belong to $L$ for all $a\in R$.
Hayes proved that there is a smallest field of definition for $\phi$ (\cite[\S 6]{Hayes}) and that, if the period lattice associated to $\phi$ is isomorphic to an invertible ideal of $R$, then the field of definition $H_R$ is a Galois extension of $K$ which is unramified at the primes outside $\mathfrak{C}$, and there is a canonical isomorphism 
\[\sigma^R:\mathrm{Cl}(R)\to\mathrm{Gal}(H_R/K)=: G_R\] such that, for any $I\in\mathcal{I}_R$,
\[\phi_a^{[I]}={\sigma^R_{[I]}}(\phi_a)\;\forall a\in R,\]
where ${\sigma^R_{[I]}}$ acts coefficient-wise (\cite[Th. 8.5, Thm. 8.8]{Hayes}). By analogy with the classical theory, the field $H_R$ is called the \emph{ring class field} attached to the order $R$.

Under the previous hypotheses, Hayes also proved that the map
\[\begin{tikzcd}[row sep=small]
    &\mathcal{I}_R\arrow[r]& H_R[G_R][\tau]\\
    &I\arrow[r,mapsto]&\sigma^R_{[I]^{-1}}\phi_I
\end{tikzcd}\]
is multiplicative (\cite[Thm. 3.10]{Hayes}), and deduced that, for any $I\in \mathcal{I}_R$, $\partial\phi_I B_R=I B_R$.

\vspace{5mm}
Let $H_{R,\infty}:= H_R\otimes_K K_\infty$, and define the ring $B_R\hat{\otimes}H_{R,\infty}$ as the completion of $B_R\otimes_{\FF_q}H_{R,\infty}$ with respect to the topology induced by $H_{R,\infty}$. We also denote by $B_R\hat{\otimes}H_{R,\infty}[G_R]$ (resp. $B_R\otimes_{F_q} H_R[G_R]$) the noncommutative ring generated by $B\hat{\otimes}H_{R,\infty}$ (resp. $B_R\otimes_{\FF_q}H_R$) and by $G_R$, with the relation \[{\sigma^R_{[I]}}(b\otimes c)=({\sigma^R_{[I]}}(b)\otimes{\sigma^R_{[I]}}(c)){\sigma^R_{[I]}}.\]

In \cite[\S 3.3]{ANG&NGO&TAV}, Anglès, Ngo Dac, and Tavares Ribeiro study an \emph{equivariant $A$-harmonic series} in $H_{A,\infty}[G_R]$, whose determinant is the value at $s=1$ of Goss' version of a Dedekind-like zeta. In the case $R=A$, with $\phi$ a Drinfeld $A$-module of rank $1$, it is defined as follows:

\[\mathcal{L}(\phi):=\prod_{P\in\mathrm{Specm}(A)}(1-\partial\phi_P^{-1}{\sigma^A_P})^{-1}=\sum_{0\neq I\subseteq A}\partial\phi_I^{-1}{\sigma^A_I}\in H_{A,\infty}[G_A].\]
The infinite product is well defined because all the factors commute with one another by Hayes; it's easy to prove that the product converges because for any valuation $w$ on $H_A$ above infinity, \[\lim_{\substack{I\subseteq A\\\deg(I)\to\infty}}w(\partial\phi_I)=-\infty.\]

We introduce a slight modification of this zeta, for an arbitrary order $R$. Let $\phi$ be a Drinfeld $R$-module whose period lattice is isomorphic to an invertible ideal of $R$; we set:

\[\xi(\phi):=\prod_{\substack{P\in\mathrm{Specm}(A)\\P\nmid\mathfrak{C}}}(1-(\partial\phi_P\otimes\partial\phi_P^{-1}){\sigma^R_P})^{-1},\]
which is well defined and converges in $B\hat\otimes H_{R,\infty}[G]$ by the same reasoning as the one used by Anglès, Ngo Dac, and Tavares Ribeiro; we can also similarly rearrange this product to get an infinite series indexed over the ideals in $\mathcal{I}_R$:
\[\xi(\phi)=\sum_{I\in\mathcal{I}_R}\partial\phi_I\otimes\partial\phi_I^{-1}\sigma^R_{[I]}\in B_R\hat\otimes H_{R,\infty}[G_R].\]

Let's consider the natural $K\otimes_{\FF_q} K$-linear left action of $H_R\otimes_{\FF_q} H_R[G_R]$ on $H_R\otimes_{\FF_q} H_R$, and specifically in the endomorphism \[(\partial\phi_P\otimes\partial\phi_P^{-1})\sigma^R_P=(\partial\phi_P\sigma^R_P)\otimes(\partial\phi_P^{-1}\sigma^R_P)\] for a given prime ideal $P\in \mathcal{I}_R$.
% Given a ring $S$, an $S$-module $M$, and an endomorphism $f\in\End_S(M)$, we denote by $\mathrm{Eig}_S(f)$ its eigenvalues counted with multiplicity (meaning that $\mathrm{Eig}_S(f)$, The eigenvalues of this endomorph:
% \[\mathrm{Eig}((\partial\phi_P\otimes\partial\phi_P^{-1})\sigma^R_P)=\left\{\lambda\otimes\mu\mid\lambda\in \mathrm{Eig}(\partial\phi_P\sigma^R_P),\mu\in \mathrm{Eig}(\partial\phi_P^{-1}\sigma^R_P)\right\}.\]

% The product of its eigenvalues (counted with multiplicity), is given by 
% \[\det(1-(\partial\phi_P\otimes\partial\phi_P^{-1})\sigma^R_P)=\left(\det(1-\partial\phi_P\sigma^R_P)\otimes\det(1-\partial\phi_P^{-1}\sigma^R_P)\right)^{h_R},\]

% where $h_R:=[H_R:K]$.

If $e$ is the order of ${\sigma^R_P}$ in $G_R$, Hayes' results tell us that $e$ is the order of $P$ in $\mathrm{Cl}(R)$, that $c_{P^e}:=\partial\phi_{P^e}\in R$ is a generator of the ideal $P^e\subseteq R$, and that \[(\partial\phi_P^{\pm1}\sigma^R_P)^e=c_{P^e}^{\pm1}\cdot\mathrm{id}_{H_R}.\]

Let's prove the following lemma. Note that its proof is somewhat complicated by the fact that $e$ is not necessarily coprime with the characteristic of the base field.
\begin{Lemma}
    The characteristic polynomial of $(\partial\phi_P\otimes\partial\phi_P^{-1})\sigma^R_P$ is
\[(T^e-c_{P^e}\otimes c_{P^e}^{-1})^\frac{h_R^2}{e},\]
where $h_R:=\dim_K(H_R)=\#\mathrm{Cl}(R)$.
\end{Lemma}

\begin{proof}
Consider $\partial\phi_P^{\pm1}\sigma^R_P$ as a $K$-linear endomorphism of $H_R$. We can prove the following:
\begin{itemize}
    \item the minimal polynomial of $\partial\phi_P^{\pm1}\sigma^R_P$ is $T^e-c_{P^e}^{\pm1}$;
    \item the characteristic polynomial of $\partial\phi_P^{\pm1}\sigma^R_P$ is $(T^e-c_{P^e}^{\pm1})^\frac{h_R}{e}$.
\end{itemize}

For the first statement we already observed that the minimal polynomial of $\partial\phi_P^{\pm1}\sigma^R_P$ in $K[X]$ divides $X^e-c_{P^e}^{\pm1}$. On the other hand, a nontrivial $H_R$-linear combination of powers of $\sigma^R_P$ smaller than $e$ is never $0$ because the canonical map $H_R[G_R]\to\mathrm{End}_K(H_R)$ is bijective. 

We can prove the second statement without loss of generality by base changing from $K$ to $H_R$. Since $H_R/K$ is a separable extension, the trace map induces a canonical isomorphism between $H_R\otimes_K H_R$ and $\mathrm{End}_K(H_R)$; moreover, the latter is canonically isomorphic to $H_R[G_R]$. It's easy to check that the induced action of $\partial\phi_P^{\pm1}\sigma^R_P$ on $H_R\otimes_K H_R$ corresponds to the left multiplication by $\partial\phi_P^{\pm1}\sigma^R_P$ on $H_R[G_R]$. Denote by $\{G_i:=\langle\sigma^R_P\rangle\tau_i\}_{i=1,\dots,h_R/e}$ the right cosets of the cyclic subgroup $\langle\sigma^R_P\rangle\subseteq G_R$, so that $H_R[G_R]=\oplus_{i=1}^{h_R/e}H_R[G_i]$. The action of $\partial\phi_P^{\pm1}\sigma^R_P$ preserves each direct summand $H_R[G_i]$, and its characteristic polynomial on $H_R[G_i]$ is $T^e-c_{P^e}^{\pm1}$. We deduce that the characteristic polynomial of $\partial\phi_P^{\pm1}\sigma^R_P$ as a $K$-linear endomorphism of $H_R$ is $(T^e-c_{P^e}^{\pm1})^\frac{h_R}{e}$.

With some simple linear algebra, it's possible to show that, if $f$ and $g$ are two endomorphisms of a finite-dimensional vector space $V$ over a field $E$ with a subfield $L$, with eigenvalues $\{\mu_i\}_i$ and $\{\nu_j\}_j$ respectively (counted with multiplicity), then the characteristic polynomial of $f\otimes g$ as a $E\otimes_L E$-linear endomorphism of $V\otimes_L V$ is $\prod_{i,j}(T-\mu_i\otimes\nu_j)$.

This tells us that the characteristic polynomial of $(\partial\phi_P\otimes\partial\phi_P^{-1})\sigma^R_P=(\partial\phi_P\sigma^R_P)\otimes(\partial\phi_P^{-1}\sigma^R_P)$ as a $K\otimes_{\FF_q}K$-linear endomorphism of $H_R\otimes_K H_R$ is 
\[\prod_{i,j}(T-\mu_i\otimes\nu_j)=\prod_j(T^e-c_{P^e}\otimes\nu_j^e)^\frac{h_R}{e}=\prod_j(T^e-c_{P^e}\otimes c_{P^e}^{-1})^\frac{h_R}{e}=(T^e-c_{P^e}\otimes c_{P^e}^{-1})^\frac{h_R^2}{e}.\]
\end{proof}

We use the characteristic polynomial of $(\partial\phi_P\otimes\partial\phi_P^{-1})\sigma^R_P$ to obtain the following expression for the determinant of $\xi(\phi)$ as a $K\hat\otimes K_\infty$-linear endomorphism of $H_R\hat\otimes H_{R,\infty}$:
\begin{align*}
    &\det(\xi(\phi))=\prod_{\substack{P\in\mathrm{Specm}(R)\\P\nmid\mathfrak{C}}}\det(1-(\partial\phi_P\otimes\partial\phi_P^{-1}){\sigma^R_P})^{-1}=\prod_{e|h_A}\prod_{\substack{P\in\mathrm{Specm}(A)\\P\nmid\mathfrak{C}\\\mathrm{ord}_{\mathrm{Cl}(R)}([P])=e}}(1-c_{P^e}\otimes c_{P^e}^{-1})^{-\frac{h_R^2}{e}}\\
    =&\prod_{\substack{Q\in\mathrm{Specm}(B)\\Q\nmid\mathfrak{C}}}(1- N_{H/K}(Q)\otimes N_{H/K}(Q)^{-1})^{-h_R}=\left(\sum_{\substack{J\subseteq B\\J+\mathfrak{C}B=B}} N_{H/K}(J)\otimes N_{H/K}(J)^{-1}\right)^{h_R},
\end{align*}
where
% we identified with slight abuse of notations the primes of $R$ away from $\mathfrak{C}$ with the primes of $A$ away from $\mathfrak{C}$, and 
we used that for any such ideal $P$, the order of $[P]$ in $\mathrm{Cl}(R)$ is equal to the number of distinct prime ideals in $B$ above $P$, and all such primes have norm $P^e$.

By Theorem \ref{zeta-E-A-are-entire} we already know that the Dedekind-like series $\zeta_{H_R,A}$ is analytic on $X_{\CC_\infty}\setminus\infty$; we are able to explicitly describe its divisor outside $\infty$ up to a (finite) effective divisor.

\begin{Theorem}
    The series $\zeta_{H_R,A}$ is analytic on $X_{\CC_\infty}\setminus\{\infty\}$, and its divisor outside $\infty$ is
    \[\mathrm{Div}^+(\zeta_{H_R,A})=W+h_R\sum_{i\geq1}\Xi^{(i)},\]
    where $W$ is an effective $K_\infty$-rational divisor of finite degree and $\Xi^{(i)}\in X(K)$ corresponds to the ring homomorphism $A\to K$ sending $a\in A$ to $a^{q^i}$.
\end{Theorem}
\begin{proof}
Fix representative ideals $\{I_\sigma\}_{\sigma\in G_R}$ of $R$ for each ideal class in $\mathrm{Cl}(R)$, and define the rational functions
\[Z_{X,I_\sigma^{-1},n}:=-\sum_{\substack{a\in I_\sigma^{-1}\setminus\{0\}\\\deg(a)\leq n}}a\otimes a^{-1},\]
so that in $\TT_A$ we have \[\lim_n Z_{X,I_\sigma^{-1},n}=Z_{X,I_\sigma^{-1}}=-\sum_{a\in I_\sigma^{-1}\setminus\{0\}}a\otimes a^{-1}.\]
Denote by $d$ the degree of the ideal $I_\sigma A\subseteq A$ and by $c$ the dimension of the $\FF_q$-vector space $A/R$; by Riemann-Roch, for $n\gg0$, the elements of degree at most $n$ in $I_\sigma^{-1}$ form an $\FF_q$-vector space of dimension $n-d-c-g+1$; in particular, for $1\leq m\leq n-d-c-g$ we have
\[Z_{X,I_\sigma^{-1},n}(\Xi^{(m)})=\sum_{\substack{a\in I_\sigma^{-1}\setminus\{0\}\\\deg(a)\leq n}}a^{q^m-1}=0.\]
Since the poles of $Z_{X,I_\sigma^{-1},n}$ outside $\infty$ coincide with the effective divisor $D_{I_\sigma}$ associated to the ideal $I_\sigma A\subseteq A$, we have:
\[\Div_X(Z_{X,I_\sigma^{-1},n})=W_{\sigma,n}-D_{I_\sigma}+\sum_{i=1}^{n-d-c-g}\Xi^{(i)}-n\infty,\]
where $W_{\sigma,n}$ is an effective divisor of degree $c+g$. Using the same techniques as \cite{FER1}, this allows us to prove that the divisor of the limit $Z_{X,I_\sigma^{-1}}$ outside $\infty$ is:
\[\Div^+_X(Z_{X,I_\sigma^{-1}})=W_\sigma-D_{I_\sigma}+\sum_{i\geq1}\Xi^{(i)},\]
where $W_\sigma$ is an effective $K_\infty$-rational divisor of degree $\leq c+g$; recall that the divisor of $Z_{X,A}$ outside infinity is instead:
\[\Div^+_X(Z_{X,A})=V_*^{(1)}+\sum_{i\geq1}\Xi^{(i)},\]
where $V_*$ is an effective divisor of degree $g$ (which is a dual Drinfeld divisor). In particular, we can write
\[Z_{X,I_\sigma^{-1}}=l_\sigma Z_{X,A}\] for some rational function $l_\sigma$ on $X_{K_\infty}$.

We can rewrite 
\[\xi(\phi)=\sum_{\sigma\in G_R}(\partial\phi_{I_\sigma}\otimes\partial\phi_{I_\sigma}^{-1})Z_{X,I_\sigma^{-1}}\sigma=\left(\sum_{\sigma\in G_R}(\partial\phi_{I_\sigma}\otimes\partial\phi_{I_\sigma}^{-1})l_\sigma\sigma\right)\zeta_A,\] 
hence its determinant is a rational multiple of $\zeta_A^{h_R^2}$. 
We can check that the determinant has actually no poles outside infinity by showing that the same is true for all coefficients $(\partial\phi_{I_\sigma}\otimes\partial\phi_{I_\sigma}^{-1})Z_{X,I_\sigma^{-1}}$:
if we denote by $Y\to X$ the cover associated to the field extension $H_R/K$, since $\partial\phi_{I_\sigma}B_R=I_\sigma B_R$, the divisor of $\partial\phi_{I_\sigma}\otimes\partial\phi_{I_\sigma}^{-1}$ on $Y_{\CC_\infty}$ is
\[\Div^+_Y(\partial\phi_{I_\sigma}\otimes\partial\phi_{I_\sigma}^{-1})=D_{I_\sigma}-\deg(I_\sigma)\infty,\]
where by slight abuse of notation we denote in the same way the effective divisors $I_\sigma$ and $\infty$ on $X_{\CC_\infty}$ and their pullback to $Y_{\CC_\infty}$, hence
\[\Div_Y^+\left((\partial\phi_{I_\sigma}\otimes\partial\phi_{I_\sigma}^{-1})Z_{X,I_\sigma^{-1}}\right)=\Div_Y^+\left((\partial\phi_{I_\sigma}\otimes\partial\phi_{I_\sigma}^{-1})\right)+\Div_Y^+\left(Z_{X,I_\sigma^{-1}}\right)=W_\sigma+\sum_{i\geq1}\Xi^{(i)}.\]

To conclude the proof, we simply need to show that the determinant of $\sum_{\sigma\in G_R}(\partial\phi_{I_\sigma}\otimes\partial\phi_{I_\sigma}^{-1})l_\sigma\sigma$ has no poles in any positive twist of $\Xi$, which is obvious because the poles of any $l_\sigma$, namely $V_*^{(1)}+D_{I_\sigma}$, do not contain $\infty$.
\end{proof}

\section{Computations}

Fix $R:= A$, $B:= B_A$, $H:= H_A$.
For a given positive degree $d$, let's denote by $\mathcal{I}(d)$ the set of ideals in $B$ whose norm has degree $d$. We know that the finite sums
\[\zeta_{H,A}(d):=\sum_{I \in\mathcal{I}(d)}N_{H/K}(I)\otimes N_{H/K}(I)^{-1}\in A\otimes K_\infty\]
and
\[\zeta_{H,A}(\le d):=\sum_{i \le d}\zeta_{H,A}(d)\]
are well defined, and that the latter converges for $d\to\infty$ to the analytic function $\zeta_{H,A}\in \TT_A$.
Given a nonnegative integer $r$, it can be proven using \cite[Lemma 8.8.1]{GOS} that, if for all $H$-rational effective divisors $D$ of degree $d$ with support at infinity we have $\dim H^0(D-\infty)>h(1+r(q-1))$, then the valuation of $\zeta_{H,A}(d)$ is at least $d+q^r$. By Riemann--Roch, since $\deg(D-\infty)=d-h$, the first inequality is implied by $d-h-g_H+1>h(1+r(q-1))$, where $g_H$ is the genus of $H$. By Riemann--Hurwitz, since the extension $H/K$ is unramified, we can write $2-2g_H=h(2-2g)$, hence the equivalent inequalities 
\[r<\frac{\frac{d}{h}-g-1}{q-1}\Leftrightarrow d>h(g+1+r(q-1))\]
imply that the valuation of $\zeta_{H,A}(d)$ is at least $d+q^r$.

In other words, the fastest approximations of $\zeta_{H,A}$ are obtained for small $q$'s and $h$'s; nevertheless, even for the optimal conditions $q=3,h=2,g=1$, we would expect the valuation of $\zeta_{H,A}(d)$ to exceed e.g. $100$ only for $d>20$. It turns out that, at least for the examples we used for our computation, this is not the case, and the approximations converge much faster. In particular, it looks like the valuation of $\zeta_{H,A}(d)$ grows approximately like $q^{\frac{d}{h}-g}$, rather than its $q-1$-th root.
%(the inequality $h \ge \frac{(q-1)(q^{2g}-2gq^g+1)}{2g(q^{g+1}-1)}$ tells us that large $q$'s also have large $h$'s).

By reason of the following computations, we state the following conjecture:
\begin{Conjecture}\label{multiplicity-conjecture} The analytic function $\zeta_{H,A}$ has a zero of multiplicity $h-1$ at $\Xi$.
\end{Conjecture}

For simplicity, we have tested elliptic and hyperelliptic curves defined by equations of the form $y^2=D(x)$, with $q=3,5$. The product between $\zeta_{H,A}(\le d)$ and its conjugate is a polynomial $N_{H/K}(\zeta_{H,A}(\le d))\in K_\infty[x]$ of degree $d$, and we use its factorization to deduce properties of the zeroes of $\zeta_{H,A}$. Call $u$ a uniformizer of $K_\infty$, and set the image of $x$ in $K_\infty$ to be $t:=u^{-2}$.
\begin{itemize}
    \item$\boxed{q=3,h=2,g=1;\;y^2=x^3-x^2-x}$. The valuations of $\zeta_{H,A}(d)$ for $d=1\dots10$ are:
    \[[\infty,2,4,6,12,18,36,54,108,162]\]
    ---they seem to follow the formula $2\cdot 3^{\frac{d-2}{2}}$ when $d$ is even and $4\cdot 3^{\frac{d-3}{2}}$ when $d$ is odd.
    If we truncate the coefficients of $N_{H/K}(\zeta_{H,A}(\le 10))$ at $O(u^{324})$ (the same as the expected valuation of $\zeta_{H,A}(11)$), up to a scalar factor it coincides with the product
    \[\left(x-t\frac{t^2+t+1}{t^2-t+1}\right)(x-t)(x-t^3)^2(x-t^9)^2(x-t^{27})^2(x-t^{81})^2.\]

    \item$\boxed{q=5,h=2,g=1;\;y^2=x^3+2x}$. The valuations of $\zeta_{H,A}(d)$ for $d=1\dots8$ are:
    \[[\infty,2,6,10,30,50,150,250]\]
    ---they seem to follow the formula $2\cdot 5^{\frac{d-2}{2}}$ when $d$ is even and $6\cdot 5^{\frac{d-3}{2}}$ when $d$ is odd. If we truncate the coefficients of $N_{H/K}(\zeta_{H,A}(\le 8))$ at $O(u^{750})$ (the same as the expected valuation of $\zeta_{H,A}(9)$), up to a scalar factor it coincides with the product
    \[\left(x-t\left(\frac{t^2-1}{t^2+1}\right)^2\right)(x-t)(x-t^5)^2(x-t^{25})^2(x-t^{125})^2.\]
    
    \item$\boxed{q=3,h=3,g=1;\;y^2=x^3+x^2-x+1}$. Since $q|h$, for combinatorial reasons $\zeta_{H,A}(d)$ is nonzero if and only if $3|d$. For $d=1,\dots,15$, the valuations of $\zeta_{H,A}(d)$ are:
    \[[\infty,\infty,3,\infty,\infty,9,\infty,\infty,27,\infty,\infty,81,\infty,\infty,243]\]
    ---they seem to follow the formula $3^{\frac{d}{3}}$ (when $3|d$).
    The coefficients of $N_{H/K}(\zeta_{H,A}(\le15))$ have valuations 
    \[[ 0, 4, 6, 6, 22, 24, 24, 76, 78, 78, 238, 240, 240, 724, 726, 726 ].\]
    If we truncate the polynomial at the $15$-th degree and each coefficient at $O(u^{720})$, the polynomial can be decomposed---up to a scalar factor---as \[(x-(t-1+t^{-1}+t^{-2}))(x-t)^2(x-t^3)^3(x-t^9)^3(x-t^{27})^3(x-t^{81})^3.\]
    
    \item$\boxed{q=3,h=3,g=2;\;y^2=x^5+x^3+x-1}$. For $d=1,\dots,18$, the valuations of $\zeta_{H,A}(d)$ are:
    \[[\infty,\infty,\infty,\infty,\infty,6,\infty,\infty,12,\infty,\infty,30,\infty,\infty,84,\infty,\infty,246]\]
    ---they seem to follow the formula $3^{\frac{d-3}{3}}+3$ (when $3|d$). The coefficients of $N_{H/K}(\zeta_{H,A}(\le18))$ have valuations 
    \[[ 0, 6, 6, 6, 16, 12, 12, 34, 30, 30, 88, 84, 84, 250, 246, 246, 736, 732, 732].\]
    If we truncate the polynomial at the $15$-th degree and each coefficient at $O(u^{728})$, the polynomial can be decomposed as \[f(x-t)^2(x-t^3)^3(x-t^9)^3(x-t^{27})^3,\] where $f$
    is a polynomial of degree $4$ with only one root in $K_\infty$. Differently from previous cases, the coefficients of $f$ don't seem to be rational in $t$; by direct computation, if they were, they would have degree at least $8$.
\end{itemize}

Let's call $v\in K_\infty$ the image of $y\in K$ under the canonical inclusion $\Xi$. In all the examples above, it can be checked that the evaluation of $\zeta_{H,A}$ at $x=t$ and $y=-v$ is nonzero, implying that if $N_{H/K}(\zeta_{H,A})$ has a zero of some multiplicity at $x=t$, then $\zeta_{H,A}$ has a zero of the same multiplicity at $(x,y)=(t,v)$.

If we assume the conjecture to hold, we deduce that the divisor $D$ of ``nontrivial'' zeros has degree at most $h (g-1)+1$, and that $D^{(1)}-D$ is linearly equivalent to $(h-1)\Xi+\Xi^{(1)}-h\infty$. 

When $g=1$, i.e. $X$ is an elliptic curve, we have $\deg(D)=g_H=g=1$; since given a dual shtuka divisor $V^*$ we have ${V^*}^{(1)}-V^*\sim\Xi-\infty$, and $D^{(-1)}$ satisfies the same identity, we deduce that $D$ is the twist of a dual shtuka divisor.

\end{document}